\documentclass[a4paper]{amsart}

\usepackage{amsmath}
\usepackage{amssymb}
\usepackage{amsthm}
\usepackage{enumitem}
\usepackage{float}
\usepackage{mathtools}
\usepackage{tikz-cd}
\usepackage{tikz}
\usetikzlibrary{arrows.meta,decorations.markings,tikzmark}
\usepackage{xcolor}

\usepackage[hidelinks]{hyperref}

\newtheorem*{thm*}{Theorem}
\newtheorem{step}{Step}
\newtheorem{thm}{Theorem}[section]
\newtheorem{lem}[thm]{Lemma}

\theoremstyle{definition}
\newtheorem{df}[thm]{Definition}
\newtheorem{ex}[thm]{Example}
\theoremstyle{remark}
\newtheorem{rmk}[thm]{Remark}

\makeatletter
\g@addto@macro\bfseries{\boldmath}
\makeatother

\numberwithin{equation}{section}

\DeclareMathOperator{\Aff}{Aff}   
\DeclareMathOperator{\Aut}{Aut}   
\DeclareMathOperator{\Coin}{Coin} 
\DeclareMathOperator{\Fix}{Fix}   
\DeclareMathOperator{\fix}{fix}   
\DeclareMathOperator{\GL}{GL}	  
\DeclareMathOperator{\id}{id}	  

\newcommand{\CC}{\mathbb{C}}      
\newcommand{\QQ}{\mathbb{Q}}      
\newcommand{\RR}{\mathbb{R}}      
\newcommand{\ZZ}{\mathbb{Z}}      

\newcommand{\n}{\mathfrak{n}}     
\newcommand{\NR}{\mathcal{NR}}    
\newcommand{\orb}{\backslash}     

\title{Nielsen numbers of $n$-valued maps on infra-solvmanifolds}
\author{Karel Dekimpe and Lore De Weerdt}

\thanks{Research supported by Methusalem grant METH/21/03 -- long term structural funding of the Flemish Government.}
\address{KU Leuven Campus Kulak Kortrijk, 8500 Kortrijk, Belgium}
\email{Karel.Dekimpe@kuleuven.be}
\email{Lore.DeWeerdt@kuleuven.be}

\begin{document}

\begin{abstract}
We derive a formula for the Nielsen number $N(f)$ for every $n$-valued self-map $f$ of an infra-solvmanifold. To do this, we express $N(f)$ in terms of Nielsen coincidence numbers of single-valued maps on solvmanifolds, and derive a formula for Nielsen coincidence numbers in that setting.
\end{abstract}

\maketitle

\section{Introduction}
For a self-map $f:M\to M$ of a closed manifold $M$, the Nielsen number $N(f)$ is a sharp lower bound for the number of fixed points among all maps homotopic to $f$. This number contains a lot of information, but is very hard to compute from its definition. For specific classes of manifolds, however, explicit formulas for the Nielsen number have been shown, which allow one to compute $N(f)$ purely in terms of the fundamental group morphism $f_\#:\pi_1(M)\to \pi_1(M)$ induced by $f$. In particular, if $M$ is a specific type of solvmanifold (called an $\NR$-solvmanifold), E. Keppelmann and C. McCord \cite{keppelmannmccord} showed the product formula
\begin{equation}\label{eq:KeppelmannMcCord}
N(f)=\prod_{i=0}^c \left|\det(I-F_i)\right|
\end{equation}
where $F_i\in \ZZ^{k_i\times k_i}$ are matrices induced by the morphism $f_\#$, and $I$ is the identity matrix. Using this result, K. Dekimpe and I. Van den Bussche were able to show a general formula for the Nielsen number of a self-map on any infra-solvmanifold \cite{iris}, by expressing $N(f)$ as the average of the Nielsen numbers of all lifts of $f$ to an $\NR$-solvmanifold. The goal of this paper is to generalise the latter formula to the setting of $n$-valued maps. In that setting, it is no longer possible to lift maps to self-maps of an $\NR$-solvmanifold (see e.g.\ \cite[Remark 5.1]{affien}). Instead, we can average over the Nielsen \emph{coincidence} numbers of maps between (a special type of) $\NR$-solvmanifolds. We generalise the formula \eqref{eq:KeppelmannMcCord} of Keppelmann and McCord to this setting (Theorem \ref{thm:N(f,q)}), and use this to obtain a formula for the Nielsen number of an $n$-valued map on any infra-solvmanifold (Theorem \ref{thm:main}).

\section{Nielsen theory}

\subsection{Nielsen coincidence theory}\label{subsec:Nielsen-coin}

For a pair of maps $f,g:M\to M'$ between closed manifolds of equal dimension, we can consider the set of coincidence points \[
\Coin(f,g)=\{x\in M \mid f(x)=g(x) \}.
\]
If $M$ and $M'$ are triangulable (in particular if they are smooth, which will be the case for all manifolds we consider) and orientable, the Nielsen coincidence number $N(f,g)$ is defined by Schirmer \cite{schirmer-coin} as follows. 

First, the coincidence set $\Coin(f,g)$ is partitioned into \emph{coincidence classes}, where $x,y\in \Coin(f,g)$ belong to the same coincidence class if there is a path $\omega$ from $x$ to $y$ such that the paths $f\omega$ and $g\omega$ are homotopic. Then, each coincidence class is assigned an \emph{index}, which is an integer that is zero if the coincidence class can become empty after a homotopy. The Nielsen number is the number of \emph{essential} coincidence classes, i.e.\ the ones with non-zero index. It is a lower bound for the number of coincidences among all pairs of maps homotopic to $f$ and $g$, i.e.\ \[
N(f,g)\leq \min\{\#\Coin(f',g') \mid f'\simeq f,g'\simeq g \},
\]
and the equality holds whenever the dimension of the manifolds $M$ and $M'$ is at least $3$ (see \cite{schirmer-coin}).

We will be interested in Nielsen numbers on manifolds given in the form $\pi\orb \tilde{M}$, where $\tilde{M}$ is a simply connected manifold and $\pi$ is a group acting properly discontinuously on $\tilde{M}$. In this setting, Nielsen numbers of maps $\pi\orb \tilde{M}\to \pi'\orb \tilde{M}'$ are often expressed in terms of their induced morphisms $\pi\to \pi'$, defined as follows.

Let $M=\pi\orb \tilde{M}$ and $M'=\pi'\orb \tilde{M}'$ be as above. For every map $f:M\to M'$ we can take a lift $\tilde{f}:\tilde{M}\to \tilde{M}'$, i.e.\ a map such that $f(\pi \tilde{x})=\pi'\tilde{f}(\tilde{x})$ for all $\tilde{x}\in \tilde{M}$. This lift $\tilde{f}$ induces a group morphism $\varphi:\pi\to \pi'$ defined by \[
\tilde{f}(\gamma\tilde{x})=\varphi(\gamma)\tilde{f}(\tilde{x}),\quad \forall \gamma\in \pi,\,\forall \tilde{x}\in \tilde{M}.
\]

\begin{rmk}
Any other lift of $f$ can be written as $\alpha\tilde{f}$ for some $\alpha\in \pi'$; the morphism corresponding to that lift is $\mu(\alpha)\varphi$, where $\mu(\alpha):\pi'\to \pi':\beta\mapsto \alpha\beta\alpha^{-1}$ is conjugation with $\alpha$. 
We will sometimes abuse terminology and speak of `the morphism induced by $f$', meaning the morphism induced by \emph{some lift of} $f$, where it does not matter which lift we take. One should keep in mind that this morphism is defined up to an inner automorphism of $\pi'$.
\end{rmk}

Recall that, given a basepoint $x\in M$, for every $\tilde{x}\in\tilde{M}$ with $\pi\tilde{x}=x$ we have an isomorphism \[
\Psi_{\tilde{x}}:\pi_1(M,x)\to \pi
\]
defined as follows. For a loop $\alpha:[0,1]\to M$ based at $x$, let $\tilde{\alpha}:[0,1]\to \tilde{M}$ be the unique lift of $\alpha$ starting at $\tilde{x}$. Since $\alpha(1)=x$, there is a $\gamma\in \pi$ such that $\tilde{\alpha}(1)=\gamma \tilde{x}$. The morphism $\Psi_{\tilde{x}}$ sends the homotopy class of $\alpha$ to this element $\gamma$. For $M'$, we have isomorphisms $\Psi_{\tilde{x}'}:\pi_1(M',x')\to \pi'$ defined in a similar fashion. With this notation, we have a commutative diagram \[
\begin{tikzcd}
\pi_1(M,x) \ar[r,"f_\#"] \ar[d,"\Psi_{\tilde{x}}"'] & \pi_1(M',f(x)) \ar[d,"\Psi_{\tilde{f}(\tilde{x})}"] \\
\pi \ar[r,"\varphi"] & \pi'.
\end{tikzcd}
\]

\begin{rmk}\label{rmk:phi-hom-inv}
It follows that the induced morphism $\varphi$ is a homotopy invariant. That is, for every map $f'$ homotopic to $f$, there is some lift of $f'$ for which the induced morphism is also $\varphi$.
\end{rmk}

\subsection{Nielsen fixed point theory of $n$-valued maps}

An $n$-valued map of a closed manifold $M$ is a continuous multifunction $f:M\multimap M$ such that $f(x)\subseteq M$ has cardinality exactly $n$ for each $x\in M$ (see \cite{browngoncalves} for more details). The fixed point set of an $n$-valued map is defined as \[
\Fix(f)=\{x\in M\mid x\in f(x) \}.
\] 
The Nielsen number $N(f)$ is a homotopy-invariant lower bound for $\# \Fix(f)$, defined in analogy with the single-valued case by Schirmer in \cite{schirmer2}.

An equivalent and more practical way to view $n$-valued maps is described in \cite{browngoncalves}: an $n$-valued multifunction $f:M\multimap M$ is continuous if and only if the corresponding single-valued function \[
M\to D_n(M):x\mapsto f(x)
\] 
is continuous, where \[
D_n(M)=\{\{x_1,\ldots,x_n\}\subseteq M\mid x_i\neq x_j\text{ if }i\neq j \}
\] 
is the \emph{unordered configuration space}, topologised as the quotient of the ordered configuration space \[
F_n(M)=\{(x_1,\ldots,x_n)\in M^n\mid x_i\neq x_j\text{ if }i\neq j \}
\]
under the action of the symmetric group $\Sigma_n$. Thus, any $n$-valued map $f:M\multimap M$ can be viewed as a continuous function $M\to D_n(M)$, which is denoted by $f$ as well.

A theory of lifts and covering group morphisms for $n$-valued maps is developed in \cite{charlotte1}.
If $p:\tilde{M}\to M$ is the universal cover of $M$, with covering group $\pi\cong \pi_1(M)$, a covering space for $D_n(M)$ is the \emph{orbit configuration space} \[
F_n(\tilde{M},\pi)=\{(\tilde{x}_1,\ldots,\tilde{x}_n)\in \tilde{M}^n \mid p(\tilde{x}_i)\neq p(\tilde{x}_j) \text{ if } i\neq j \},
\]
with covering map \[
p^n:F_n(\tilde{M},\pi)\to D_n(M):(\tilde{x}_1,\ldots,\tilde{x}_n)\mapsto \{p(\tilde{x}_1),\ldots,p(\tilde{x}_n) \}.
\]
Any $n$-valued map $f:M\to D_n(M)$ admits a lift $\tilde{f}=(\tilde{f}_1,\ldots,\tilde{f}_n):\tilde{M}\to F_n(\tilde{M},\pi)$ such that $f(p(\tilde{x}))=\{p(\tilde{f}_1(\tilde{x})),\ldots,p(\tilde{f}_n(\tilde{x})) \}$ for all $\tilde{x}\in\tilde{M}$, \[
\begin{tikzcd}
\tilde{M} \ar[r,"\tilde{f}"] \ar[d,"p"'] & F_n(\tilde{M},\pi) \ar[d,"p^n"] \\
M \ar[r,"f"] & D_n(M).
\end{tikzcd}
\]
This lift induces maps $\varphi_1,\ldots,\varphi_n:\pi\to \pi$ and $\sigma:\pi\to \Sigma_n:\alpha\mapsto \sigma_\alpha$ such that \[
(\tilde{f}_1\alpha,\ldots,\tilde{f}_n\alpha)=(\varphi_1(\alpha)\tilde{f}_{\sigma_\alpha^{-1}(1)},\ldots,\varphi_n(\alpha)\tilde{f}_{\sigma_\alpha^{-1}(n)})
\]
for all $\alpha\in \pi$. The map $\sigma:\pi\to \Sigma_n$ is a group morphism, and so are the re\-strictions $\varphi_j:S_j\to \pi$ to the stabiliser subgroups $S_j=\{\alpha\in \pi\mid \sigma_\alpha(j)=j\}$.

In \cite{RT}, we showed that Nielsen numbers of $n$-valued maps can be expressed in terms of Nielsen coincidence numbers of single-valued maps; the following is a special case of \cite[Theorem 6.1]{RT}:

\begin{thm}\label{thm:RT}
Let $f:M\to D_n(M)$ be an $n$-valued map, and let $\varphi_j:S_j\to \pi$ be the morphisms induced by some lift $(\tilde{f}_1,\ldots,\tilde{f}_n):\tilde{M}\to F_n(\tilde{M},\pi)$ of $f$. Suppose $\bar{M}=\Gamma\orb \tilde{M}$ is an orientable finite regular covering space of $M$, such that $\pi$ has an $(f,\Gamma)$-invariant subgroup $S$ (a finite index normal subgroup of $\pi$, contained in $\Gamma$ and in all groups $S_j$, with $\varphi_j(S)\subseteq \Gamma$ for all $j$). Write $\hat{M}=S\orb\tilde{M}$, and let $\hat{p}:\hat{M}\to M$, $\bar{p}:\bar{M}\to M$ and $q:\hat{M}\to \bar{M}$ be the covering maps.

There exist maps $\bar{f}_1,\ldots,\bar{f}_n:\hat{M}\to \bar{M}$ such that \[
f(\hat{p}(\hat{x}))=\{\bar{p}(\bar{f}_1(\hat{x})),\ldots,\bar{p}(\bar{f}_n(\hat{x})) \}
\] 
for all $\hat{x}\in\hat{M}$, and such that $\tilde{f}_j:\tilde{M}\to \tilde{M}$ is a lift of $\bar{f}_j$ for every $j$. Moreover, \[
N(f)\geq \frac{1}{[\pi:S]}\sum_{j=1}^{n}\sum_{\bar{\alpha}\in \pi/\Gamma} N(\bar{\alpha}\bar{f}_j,q)
\]
where $\bar{\alpha}\in \pi/\Gamma$ acts on $\bar{M}$ as a covering translation. The inequality is an equality if $\fix(\mu(\alpha)\varphi_j)=1$ for all $\alpha$ and $j$ for which $N(\bar{\alpha}\bar{f}_j,q)\neq 0$.
\end{thm}

In \cite{RT}, this theorem was used to compute the Nielsen numbers of $n$-valued maps on infra-nilmanifolds, using the known formulas for Nielsen coincidence numbers on nilmanifolds. In this paper, we will do the same for infra-solvmanifolds, by lifting to (a special type of) solvmanifolds. However, on solvmanifolds, a way to compute Nielsen coincidence numbers of single-valued maps is not known yet. The first part of this paper is concerned with computing these. We restrict to the setting needed to apply the theorem, i.e.\ to a pair $f,q:M\to M'$ where $M'$ is this special type of solvmanifold, $M$ is a finite cover of $M'$, and $q$ is the covering map.

\section{The structure of solvmanifolds}

Solvmanifolds are homogeneous spaces of connected solvable Lie groups. For us, solvmanifolds will always be compact. Such a manifold can be represented as the coset space $\Delta\orb G$ of a connected simply connected solvable Lie group $G$ by a closed cocompact subgroup $\Delta$. A solvmanifold of the type $\Gamma\orb G$ where $\Gamma$ is not just closed but discrete, and hence a lattice in $G$, is called a \emph{special} solvmanifold.

An easier type of solvmanifolds, and very useful to study solvmanifolds in general, are nilmanifolds, which are homogeneous spaces of connected \emph{nilpotent} Lie groups.

\subsection{Nilmanifolds}\label{subsec:nilmanifolds}

By the work of Mal'cev \cite{malcev}, a nilmanifold can always be represented as the coset space $\Gamma\orb N$ of a connected simply connected nilpotent Lie group $N$ by a lattice $\Gamma$ in $N$. The fundamental group of $\Gamma\orb N$ can be identified with $\Gamma$, and this group completely determines the nilmanifold (up to homeomorphism).

Mal'cev also showed that a group $\Gamma$ can be realised as the fundamental group of a nilmanifold if and only if it is finitely generated, torsion-free and nilpotent. For each such group $\Gamma$, there is a unique connected simply connected nilpotent Lie group $N$ such that $\Gamma$ is a lattice in $N$, called the \emph{Mal'cev completion} of $\Gamma$; then $\Gamma\orb N$ is the unique nilmanifold with fundamental group $\Gamma$.
If $N'$ is another connected simply connected nilpotent Lie group, then any group morphism $\Gamma\to N'$ extends uniquely to a Lie group morphism $N\to N'$ (see e.g.\ \cite[Theorem 2.11]{rhagunatan}).

Recall that a group $G$ is nilpotent if it admits a \emph{central series}, i.e.\ a series \[
G=G_1 \rhd G_2 \rhd \ldots G_c\rhd G_{c+1}=1
\] 
of subgroups with $[G_i,G]\subseteq G_{i+1}$ for all $i$. A special central series is the \emph{lower central series}, defined inductively by $\gamma_1(G)=G$ and $\gamma_{i+1}(G)=[\gamma_i(G),G]$ for $i\geq 1$. A group $G$ is nilpotent if and only if $\gamma_{c+1}(G)=1$ for some $c$. 

For a finitely generated torsion-free nilpotent group $\Gamma$, we consider another central series whose terms are the \emph{isolators} of the subgroups $\gamma_i(\Gamma)$:
\[
\Gamma_i\vcentcolon=\sqrt[\Gamma]{\gamma_i(\Gamma)}=\{g\in \Gamma \mid \exists \ell\in \ZZ_{>0}: g^\ell \in \gamma_i(\Gamma) \}.
\]
The groups $\Gamma_i$ satisfy $\Gamma_i=\Gamma\cap N_i$, where $N$ is the Mal'cev completion of $\Gamma$ and $N_i=\gamma_i(N)$. Moreover, the quotients $\Lambda_i=\Gamma_i/\Gamma_{i+1}$ are free abelian. (For a reference, see e.g.\ \cite{karel-boek}.)

Note that the groups $N_i$ are also connected simply connected nilpotent Lie groups, and the groups $\Gamma_i$ are also finitely generated torsion-free nilpotent groups.
In particular, for every $i\geq 1$, the group $N_i$ is the Mal'cev completion of $\Gamma_i$ and $\Gamma_i\orb N_i$ is a nilmanifold with fundamental group $\Gamma_i$. Moreover:

\begin{lem}\label{lem:nil-fib}
For every $i\geq 1$, the projection \[
p_i:\Gamma_i\orb N_i\to \Gamma_iN_{i+1}\orb N_i:\Gamma_in\mapsto \Gamma_iN_{i+1}n
\]
is a fiber bundle with fibers homeomorphic to $\Gamma_{i+1}\orb N_{i+1}$. The base space is a torus with fundamental group $\Lambda_i$.
\end{lem}

\begin{proof}
For the fibers, we have homeomorphisms
\[
p_i^{-1}(\Gamma_iN_{i+1}n)=\Gamma_i\orb\Gamma_iN_{i+1}n\approx\Gamma_i\orb \Gamma_iN_{i+1}\approx(\Gamma_i\cap N_{i+1})\orb N_{i+1}=\Gamma_{i+1}\orb N_{i+1}.
\]
Note that we use $\Gamma_iN_{i+1}n$ both to denote an element of the space $\Gamma_iN_{i+1}\orb N_i$ and to denote the subset $\{xn\mid x\in \Gamma_iN_{i+1} \}$ of $N_i$. Thus, the fiber over the \emph{element} $\Gamma_iN_{i+1}n$ is the projection in $\Gamma_i\orb N_i$ of the \emph{set} $\Gamma_iN_{i+1}n$, i.e.\ the set $\{\Gamma_i xn \mid x\in \Gamma_iN_{i+1} \}$.

As an explicit homeomorphism to use later, for every representative $n\in N_i$ for $\Gamma_iN_{i+1}n$ we can take the map \[
h_n:\Gamma_i\orb\Gamma_iN_{i+1}n\to \Gamma_{i+1}\orb N_{i+1}
\]
sending an element of $\Gamma_i\orb\Gamma_iN_{i+1}n$, represented as $\Gamma_in_{i+1}n$ with $n_{i+1}\in N_{i+1}$, to $\Gamma_{i+1} n_{i+1}\in \Gamma_{i+1}\orb N_{i+1}$.

For the base space, we have a homeomorphism \[
\Gamma_iN_{i+1}\orb N_i\approx(\Gamma_iN_{i+1}/N_{i+1})\orb (N_i/N_{i+1}).
\]
It is a standard result (see e.g.\ \cite[Chapter 2, Theorem 3.4]{lieIII}) that the quotient of a simply connected solvable (so in particular, nilpotent) Lie group by a connected Lie subgroup is connected and simply connected. Therefore $N_i/N_{i+1}$ is a connected simply connected abelian Lie group. On the other hand, \[
\Gamma_iN_{i+1}/N_{i+1}\cong \Gamma_i/(\Gamma_i\cap N_{i+1})=\Gamma_i/\Gamma_{i+1}
\] 
is the free abelian group $\Lambda_i$. Hence the base space is a torus with fundamental group $\Lambda_i$. In particular, we can now use that it is a manifold.

Since the projections $N_i\to \Gamma_i\orb N_i$ and $N_i\to\Gamma_iN_{i+1}\orb N_i$ of Lie groups onto quotient manifolds are surjective submersions, so is the map $p_i$. Since $\Gamma_i\orb N_i$ is compact, it follows from Ehresmann's theorem that $p_i$ is a fiber bundle.
\end{proof}

Thus, nilmanifolds can be viewed as iterated fiber bundles, starting from a torus $\Gamma_c\orb N_c$ and where each time the base space is also a torus.

Now suppose $\Gamma$ and $\Gamma'$ are two finitely generated torsion-free nilpotent groups, with $\Gamma_i$, $\Gamma'_i$, $\Lambda_i$ and $\Lambda'_i$ as above. For every group morphism $\varphi:\Gamma\to \Gamma'$, one has $\varphi(\Gamma_i)\subseteq \Gamma'_i$ for all $i$. In particular, $\varphi$ induces well-defined morphisms $\varphi_i:\Lambda_i\to \Lambda'_i$. The collection of induced morphisms $\{\varphi_i:\Lambda_i\to \Lambda'_i\}_{i\geq 1}$ is called the \emph{linearisation} of the morphism $\varphi$.

\begin{lem}\label{lem:lin-f-nil}
Let $f:\Gamma\orb N\to\Gamma'\orb N'$ be a map between two nilmanifolds, with induced morphism $\varphi:\Gamma\to \Gamma'$ as in section \ref{subsec:Nielsen-coin}, and let $\{\varphi_i:\Lambda_i\to \Lambda'_i\}_{i\geq 1}$ be the linearisation of $\varphi$. Then $f$ is homotopic to a map that induces fiber maps $f_i:\Gamma_i\orb N_i\to \Gamma'_i\orb N'_i$ of all fibrations $p_i$ and $p'_i$. If $\bar{f}_i:\Gamma_iN_{i+1}\orb N_i\to \Gamma'_iN'_{i+1}\orb N'_i$ are the induced maps on the base spaces, we have commutative diagrams \[
\begin{tikzcd}
\pi_1(\Gamma_iN_{i+1}\orb N_i) \ar[r,"(\bar{f}_i)_\#"] \ar[d] & \pi_1(\Gamma'_iN'_{i+1}\orb N'_i) \ar[d] \\
\Lambda_i \ar[r,"\varphi_i"] & \Lambda'_i
\end{tikzcd}
\]
where the vertical arrows are isomorphisms. 

For all $n\in N_i$ and for all $n'\in N'_i$ with $\bar{f}_i(\Gamma_iN_{i+1}n)=\Gamma'_iN'_{i+1}n'$, if $(f_i)_{\Gamma_iN_{i+1}n}:\Gamma_i\orb\Gamma_iN_{i+1}n\to \Gamma'_i\orb\Gamma'_iN'_{i+1}n'$ is the induced map on the fibers, and $(f'_i)_{\Gamma_iN_{i+1}n}$ is the map corresponding to $(f_i)_{\Gamma_iN_{i+1}n}$ under the homeomorphisms from the proof of Lemma \ref{lem:nil-fib}, \[
\begin{tikzcd}[column sep=5em]
\Gamma_i\orb\Gamma_iN_{i+1}n \ar[r,"(f_i)_{\Gamma_iN_{i+1}n}"] \ar[d,"h_n"'] & \Gamma'_i\orb\Gamma'_iN'_{i+1}n' \ar[d,"h'_{n'}"] \\
\Gamma_{i+1}\orb N_{i+1} \ar[r,"(f'_i)_{\Gamma_iN_{i+1}n}"] & \Gamma'_{i+1}\orb N'_{i+1},
\end{tikzcd}
\]
the morphism induced by $(f'_i)_{\Gamma_iN_{i+1}n}$ is $\mu(\gamma')\varphi'$ for some $\gamma'\in \Gamma'_i$, where $\varphi':\Gamma_{i+1}\to \Gamma'_{i+1}$ is the morphism induced by $\varphi$.
\end{lem}

\begin{proof}
The morphism $\varphi$ extends uniquely to a morphism $F:N\to N'$, which in turn defines a map $\Gamma\orb N\to \Gamma'\orb N':\Gamma n\mapsto \Gamma'F(n)$. Since the morphism induced by this map (by the lift $F$) is also $\varphi$, and since nilmanifolds are aspherical, this map is homotopic to $f$.

The morphism $F$ induces morphisms $F_i:N_i\to N'_i$ for all $i\geq 1$, each with $F_i(\Gamma_i)\subseteq \Gamma'_i$, which in turn induce well-defined maps $f_i:\Gamma_i\orb N_i\to \Gamma'_i\orb N'_i$ and $\bar{f}_i:\Gamma_iN_{i+1}\orb N_i\to \Gamma'_iN'_{i+1}\orb N'_i$. Let $\bar{f}'_i$ be the map fitting into the diagram
\[
\begin{tikzcd}
\Gamma_iN_{i+1}\orb N_i \ar[r,"\bar{f}_i"] \ar[d] & \Gamma'_iN'_{i+1}\orb N'_i \ar[d] \\
\frac{\Gamma_iN_{i+1}}{N_{i+1}}\orb \frac{N_i}{N_{i+1}} \ar[r,"\bar{f}'_i"] & \frac{\Gamma'_iN'_{i+1}}{N'_{i+1}}\orb \frac{N'_i}{N'_{i+1}}
\end{tikzcd}
\] 
where the vertical arrows are the natural homeomorphisms.
By construction, a lift for $\bar{f}'_i$ is \[
\bar{F}_i:N_i/N_{i+1}\to N'_i/N'_{i+1}:n_iN_{i+1}\mapsto F_i(n_i)N'_{i+1}.
\] 
For all $\gamma_i\in \Gamma_i$ and $n_i\in N_i$, we have \begin{align*}
\bar{F}_i((\gamma_iN_{i+1})(n_iN_{i+1}))&=\bar{F}_i((\gamma_in_i)N_{i+1})\\&=F_i(\gamma_in_i)N'_{i+1}\\&=\varphi(\gamma_i)F_i(n_i)N'_{i+1}\\&=(\varphi(\gamma_i)N'_{i+1})(F_i(n_i)N'_{i+1})\\&=(\varphi(\gamma_i)N'_{i+1})\bar{F}_i(n_iN_{i+1})
\end{align*}
so the morphism $\Gamma_iN_{i+1}/N_{i+1}\to \Gamma'_iN'_{i+1}/N'_{i+1}$ induced by this lift is given by $\gamma_iN_{i+1}\mapsto \varphi(\gamma_i)N'_{i+1}$. Under the identifications $\Gamma_iN_{i+1}/N_{i+1}\cong \Gamma_i/\Gamma_{i+1}=\Lambda_i$ and $\Gamma'_iN'_{i+1}/N'_{i+1}\cong \Gamma'_i/\Gamma'_{i+1}=\Lambda'_i$, this is precisely the morphism $\varphi_i:\Lambda_i\to \Lambda'_i$ induced by $\varphi$.

Together, for each $\Gamma_iN_{i+1}n_i\in \Gamma_iN_{i+1}\orb N_i$, we get the diagram 
\[
\begin{tikzcd}
\pi_1(\Gamma_iN_{i+1}\orb N_i,\,\Gamma_iN_{i+1}n_i) \ar[r,"(\bar{f}_i)_\#"] \ar[d,"(h_i)_\#"'] & \pi_1(\Gamma'_iN'_{i+1}\orb N'_i,\,\Gamma'_iN'_{i+1}F_i(n_i)) \ar[d,"(h'_i)_\#"] \\
\pi_1(\frac{\Gamma_iN_{i+1}}{N_{i+1}}\orb \frac{N_i}{N_{i+1}},\frac{\Gamma_iN_{i+1}}{N_{i+1}}n_iN_{i+1}) \ar[r,"(\bar{f}'_i)_\#"] \ar[d,"\Psi_{n_iN_{i+1}}"'] & \pi_1(\frac{\Gamma'_iN'_{i+1}}{N'_{i+1}}\orb \frac{N'_i}{N'_{i+1}},\frac{\Gamma'_iN'_{i+1}}{N'_{i+1}}\bar{F}_i(n_iN_{i+1})) \ar[d,"\Psi_{\bar{F}_i(n_iN_{i+1})}"] \\
\frac{\Gamma_iN_{i+1}}{N_{i+1}} \ar[r,"{\gamma_iN_{i+1}\,\mapsto\,\varphi(\gamma_i)N'_{i+1}}"] \ar[d] & \frac{\Gamma'_iN'_{i+1}}{N'_{i+1}} \ar[d] \\
\Lambda_i \ar[r,"\varphi_i"] & \Lambda'_i
\end{tikzcd}
\]
where $\Psi_{n_iN_{i+1}}$ and $\Psi_{\bar{F}_i(n_iN_{i+1})}$ are the isomorphisms from section \ref{subsec:Nielsen-coin}.

For the second part, note that $\Gamma'_iF(n)=f_i(\Gamma_in)\in \Gamma'_iN'_{i+1}n'$, so there is a $\gamma'\in \Gamma'_i$ such that $\gamma'F(n)\in N'_{i+1}n'$. Then the map \[
F'_{i+1}:N_{i+1}\to N'_{i+1}:n_{i+1} \mapsto \gamma'F(n_{i+1}n)(n')^{-1}
\]
is well-defined and continuous, and one easily checks that this is a lift for $(f'_i)_{\Gamma_iN_{i+1}n}$. The induced morphism is $\mu(\gamma')\varphi'$, since for all $\gamma_{i+1}\in \Gamma_{i+1}$ and $n_{i+1}\in N_{i+1}$, \begin{align*}
F'_{i+1}(\gamma_{i+1}n_{i+1})&=\gamma'F(\gamma_{i+1}n_{i+1}n)(n')^{-1} \\ 
&=\gamma'\varphi(\gamma_{i+1})F(n_{i+1}n)(n')^{-1} \\
&=\gamma'\varphi(\gamma_{i+1})(\gamma')^{-1}F'_{i+1}(n_{i+1}). \qedhere
\end{align*}
\end{proof}

\subsection{Solvmanifolds}\label{subsec:solvmanifolds}

As in the case for nilmanifolds, a solvmanifold is completely determined by its fundamental group. The fundamental group of a solvmanifold $\Delta\orb G$ can be identified with the group $\Gamma=\Delta/\Delta_0$, where $\Delta_0$ is the connected component of the identity in $\Delta$. In particular, the fundamental group of a special solvmanifold $\Gamma\orb G$ is isomorphic to $\Gamma$.

A group $\Gamma$ can be realised as the fundamental group of a solvmanifold if and only if it fits into a short exact sequence \begin{equation}\label{eq:gamma-ses}
1\to \Gamma_1\to \Gamma\to \Lambda_0\to 1
\end{equation}
where $\Gamma_1$ is finitely generated torsion-free nilpotent and $\Lambda_0$ is free abelian (see \cite{wang}). Such a group $\Gamma$ is called a \emph{strongly torsion-free $S$-group}. 

\begin{rmk}\label{rmk:polycyclic}
If $\Gamma_i=\sqrt[\Gamma_1]{\gamma_i(\Gamma_1)}$ are the terms of the isolated lower central series for $\Gamma_1$, we get a series \[
\Gamma=\Gamma_0\rhd \Gamma_1\rhd \ldots \rhd \Gamma_{c+1}=1
\]
where all quotients $\Gamma_i/\Gamma_{i+1}$ are free abelian, so in particular, the group $\Gamma$ is polycyclic.
\end{rmk}

If $\Gamma$ is a strongly torsion-free $S$-group with $\Gamma_1$ as above, then $\Gamma_1$ contains the group \[
\sqrt[\Gamma]{[\Gamma,\Gamma]}=\{g\in \Gamma \mid \exists \ell\in \ZZ_{>0}: g^\ell \in [\Gamma,\Gamma] \} 
\]
since $\Gamma/\Gamma_1$ is abelian and torsion-free. On the other hand, $\sqrt[\Gamma]{[\Gamma,\Gamma]}$ is itself a finitely generated torsion-free nilpotent group (being a subgroup of $\Gamma_1$) for which the quotient $\Gamma/\sqrt[\Gamma]{[\Gamma,\Gamma]}$ is free abelian (it is finitely generated, and it is abelian and torsion-free by construction). Thus, we can take $\Gamma_1=\sqrt[\Gamma]{[\Gamma,\Gamma]}$, and this is the minimal nilpotent subgroup $\Gamma_1$ of $\Gamma$ leading to a short exact sequence of the form \eqref{eq:gamma-ses} with $\Lambda_0$ free abelian.

From now on, we will always take $\Gamma_1=\sqrt[\Gamma]{[\Gamma,\Gamma]}$. Let $\Lambda_0=\Gamma/\Gamma_1$ denote the free abelian quotient, and $\Gamma_i=\sqrt[\Gamma_1]{\gamma_i(\Gamma_1)}$ the terms of the isolated lower central series of $\Gamma_1$, with free abelian quotients $\Lambda_i=\Gamma_i/\Gamma_{i+1}$. Then $\Lambda_0$ acts on the groups $\Lambda_i$ by conjugation: for all $i\geq 1$ and $\gamma\Gamma_1\in \Lambda_0$, \[
\mu_i(\gamma\Gamma_1):\Lambda_i\to \Lambda_i: \gamma_i\Gamma_{i+1}\mapsto \gamma\gamma_i\gamma^{-1}\Gamma_{i+1}
\]
is a well-defined automorphism independent of the chosen representative $\gamma$. This gives a group morphism $\mu_i:\Lambda_0\to \Aut(\Lambda_i)$ for each $i$. The collection $\{\Lambda_i,\mu_i \}_{i\geq 1}$ is called the \emph{linearisation of $\Gamma$} (see \cite{keppelmannmccord}).

Since the groups $\Lambda_i$ are free abelian, say $\Lambda_i\cong \ZZ^{k_i}$, we can view the morphisms $\mu_i$ as morphisms $\ZZ^{k_0}\to \GL_{k_i}(\ZZ)$. Note that, for $v\in \ZZ^{k_0}$, the eigenvalues of $\mu_i(v)$ are independent the chosen identification $\Lambda_i\cong \ZZ^{k_i}$.

\begin{df}[see \cite{iris}]\label{def:net}
In the above setting, the group $\Gamma$ is called \emph{net} if for every $i\geq 1$ and $v\in \ZZ^{k_0}$, the multiplicative subgroup of $\CC^*$ generated by the eigenvalues of $\mu_i(v)$ does not contain a non-trivial root of unity.
\end{df}

\begin{rmk}
A more general notion is the one of $\NR$-group from \cite{keppelmannmccord}: they define $\Gamma$ to be $\NR$ (or actually, they call a solvmanifold with such fundamental group an $\NR$-solvmanifold) if no $\mu_i(v)$ has non-trivial roots of unity as eigenvalues. We will require the stronger property of netness in order for Lemma \ref{lem:iris} to hold, which fails for $\NR$-groups in general (see \cite[Example 2.3.10]{iris-phd}).
\end{rmk}

\subsection{The Mostow fibration}

Like for nilmanifolds, the structural results for the fundamental group of a solvmanifold give rise to a fibration on the topological level.

\begin{thm}[{\cite[Theorem 1.1]{mccord1991}}]\label{thm:Mostow}
Let $M$ be a solvmanifold with fundamental group $\Gamma$. Every short exact sequence \[
1\to \Gamma_1\to \Gamma\to \Lambda_0\to 1
\] 
with $\Gamma_1$ finitely generated torsion-free nilpotent and $\Lambda_0$ free abelian, can be realised as the homotopy group sequence of a fibration where the base space is a torus and the fiber is a nilmanifold. 

A map $f:M\to M'$ between two solvmanifolds with fundamental groups $\Gamma$ and $\Gamma'$ is homotopic to a fiber map of the respective fibrations corresponding to subgroups $\Gamma_1\subseteq \Gamma$ and $\Gamma'_1\subseteq \Gamma'$ if and only if $f_\#(\Gamma_1)\subseteq \Gamma'_1$.
\end{thm}

A fibration of this type is called a \emph{Mostow fibration}. We will consider the fibration corresponding to $\Gamma_1=\sqrt[\Gamma]{[\Gamma,\Gamma]}$, which is called the \emph{minimal Mostow fibration}.
Note that in this case, the condition of the second part of the theorem is always satisfied, so any map between two solvmanifolds is homotopic to a fiber map of their minimal Mostow fibrations.

To give a more explicit description, we restrict to the case where $M=\Gamma\orb G$ is a special solvmanifold, i.e.\ $\Gamma$ is a lattice in $G$. In analogy with the results for nilmanifolds, we will show:

\begin{thm}\label{thm:solv-fib}
For a special solvmanifold $\Gamma\orb G$, the Mal'cev completion of the group $\Gamma_1=\sqrt[\Gamma]{[\Gamma,\Gamma]}$ is a subgroup $N\subseteq G$ with $\Gamma_1=\Gamma\cap N$. Moreover, $\Gamma N$ is a subgroup of $G$, and the projection \[
p:\Gamma\orb G\to \Gamma N\orb G:\Gamma g\mapsto \Gamma Ng
\]
is a fiber bundle with fibers homeomorphic to $\Gamma_1\orb N$. The base space is a torus with fundamental group isomorphic to $\Lambda_0=\Gamma/\Gamma_1$.
\end{thm}

We proceed in several steps:

\begin{step}
The Mal'cev completion $N$ of $\Gamma_1$ is a subgroup of $G$.
\end{step}

\begin{proof}
We will show that $\Gamma_1\subseteq [G,G]$. Since $[G,G]$ is a connected simply connected nilpotent Lie group, it follows that $N$ is contained in $[G,G]$, hence in $G$.

For $g\in \Gamma_1$, there is an $\ell\in \ZZ_{>0}$ such that $g^\ell\in [\Gamma,\Gamma]$, so $g^\ell\in [G,G]$. Now, being the quotient of a simply connected solvable Lie group by a connected Lie subgroup, $G/[G,G]$ is a connected simply connected Lie group. Since it is also abelian, it is torsion-free, so $g^\ell\in [G,G]$ implies $g\in [G,G]$.
\end{proof}

\begin{step}
$\Gamma_1=\Gamma\cap N$.
\end{step}

\begin{proof}
Clearly, $\Gamma_1\subseteq \Gamma\cap N$. For the converse, note that $\Gamma\cap N$ is also a lattice in $N$: it is discrete since it is contained in $\Gamma$, and $(\Gamma\cap N)\orb N$ is compact since $\Gamma_1\orb N$ is compact and $\Gamma_1\subseteq \Gamma\cap N$. It follows that $\Gamma_1$ is a finite index subgroup of $\Gamma\cap N$. But then $(\Gamma\cap N)/\Gamma_1$ is a finite subgroup of the torsion-free group $\Gamma/\Gamma_1$, so it must be trivial. It follows that $\Gamma_1=\Gamma\cap N$.
\end{proof}

\begin{step}\label{step:GammaN}
$\Gamma$ is contained in the normaliser of $N$ in $G$. (In particular, $\Gamma N$ is a group and $N$ is normal in $\Gamma N$.)
\end{step}

\begin{proof}
Pick $\gamma\in \Gamma$. Since $\Gamma_1$ is normal in $\Gamma$, conjugation with $\gamma$ defines a group morphism \[
\varphi:\Gamma_1\to \Gamma_1:\gamma_1\mapsto \gamma\gamma_1\gamma^{-1}.
\]
This morphism extends uniquely to a Lie group morphism $\tilde{\varphi}:N\to N$ of the Mal'cev completion of $\Gamma_1$. Using the inclusion $\iota:N\to [G,G]$, we get a morphism $\iota\circ \tilde{\varphi}:N\to [G,G]$.

On the other hand, since $[G,G]$ is normal in $G$, conjugation with $\gamma$ defines a group morphism $\psi:N\to [G,G]:n\mapsto \gamma n\gamma^{-1}$. Since both are extensions of $\varphi$, they must be equal. In particular, the image of $\psi$ is contained in $N$.
\end{proof}

\begin{step}\label{step:fibers}
All fibers of $p$ are homeomorphic to the nilmanifold $\Gamma_1\orb N$.
\end{step}

\begin{proof}
For each $\Gamma Ng\in \Gamma N\orb G$, we have \[
p^{-1}(\Gamma Ng)=\Gamma\orb \Gamma Ng \approx \Gamma\orb \Gamma N \approx (\Gamma\cap N)\orb N=\Gamma_1\orb N.
\]
As an explicit homeomorphism to use later, for every representative $g\in G$ for $\Gamma Ng$ we can take the map \[
h_g:\Gamma\orb \Gamma Ng\to \Gamma_1\orb N
\]
sending an element of $\Gamma\orb \Gamma Ng$, represented as $\Gamma ng$ with $n\in N$, to $\Gamma_1 n\in \Gamma_1\orb N$.
\end{proof}

\begin{step}\label{step:base-space}
The base space is a torus with fundamental group isomorphic to $\Lambda_0$.
\end{step}

\begin{proof}
We have a homeomorphism \[
h:\Gamma N\orb G\to (N\orb \Gamma N)\orb (N\orb G):\Gamma Ng\mapsto (N\orb \Gamma N)Ng.
\]
Note that we need left cosets everywhere in order for $h$ to be well-defined, since we do not know whether $N\orb G$ is a group.

It follows from Step \ref{step:GammaN} that $N\orb \Gamma N$ is a group, and we have an isomorphism \[
N\orb \Gamma N\cong \Gamma N/N\cong \Gamma/(\Gamma\cap N)=\Gamma/\Gamma_1=\Lambda_0.
\]
On the other hand, $N\orb G$ is the quotient of a simply connected solvable Lie group by a connected Lie subgroup, so it is homeomorphic to $\RR^k$ for some $k$ (see \cite[Chapter 2, Theorem 3.4]{lieIII}).

We will show that the action of $N\orb \Gamma N$ on $N\orb G$ is properly discontinuous. It then follows that the quotient $(N\orb \Gamma N)\orb (N\orb G)$ is a manifold with fundamental group isomorphic to $N\orb \Gamma N\cong\Lambda_0$. Since it is a compact quotient manifold of $\RR^k$ with free abelian fundamental group, it has to be a torus.

To show the action of $N\orb \Gamma N$ on $N\orb G$ is properly discontinuous, suppose $\overline{K}\subseteq N\orb G$ is a compact set. We must show that $\{\bar{\gamma}\in N\orb \Gamma N \mid \bar{\gamma}\overline{K}\cap \overline{K}\neq \varnothing \}$ is finite.

Let $q:G\to N\orb G$ denote the projection. We can take a compact set $K\subseteq G$ such that $q(K)=\overline{K}$. Since $\Gamma_1\orb N$ is compact, we can also take a compact set $C\subseteq N$ with $1\in C$ whose image under the projection $N\to \Gamma_1\orb N$ is $\Gamma_1\orb N$. Then $C$ is a compact subset of $G$ with $\Gamma_1C=N$. Note that $CK$ is also compact, being the image of $C\times K$ under the (continuous) product map $G\times G\to G$.

Now suppose $\bar{\gamma}\in N\orb \Gamma N$ satisfies $\bar{\gamma}\overline{K}\cap \overline{K}\neq \varnothing$, and take $\gamma\in \Gamma$ such that $\bar{\gamma}=q(\gamma)$. Then we have $q(\gamma K)\cap q(K)\neq \varnothing$, which means \[
\varnothing\neq \gamma K\cap NK=\gamma K\cap \Gamma_1CK.
\]
Thus, there is a $\gamma_1\in \Gamma_1$ such that $\gamma K\cap \gamma_1CK\neq \varnothing$. Since $1\in C$, it follows that $\gamma CK\cap \gamma_1CK\neq \varnothing$, so $\gamma_1^{-1}\gamma CK\cap CK\neq \varnothing$. But since $CK$ is compact and the action of $\Gamma$ on $G$ is properly discontinuous, the set $F\vcentcolon=\{\gamma\in \Gamma \mid \gamma CK\cap CK\neq \varnothing \}$ is finite.
Thus, $\bar{\gamma}=q(\gamma_1^{-1}\gamma)$ belongs to the finite set $q(F)$.
\end{proof}

\begin{step}
$p$ is a fiber bundle.
\end{step}

\begin{proof}
Now that we know the base space is a manifold, we can apply the same reasoning as in Lemma \ref{lem:nil-fib}: the projections $G\to \Gamma\orb G$ and $G\to \Gamma N\orb G$ of Lie groups onto quotient manifolds are surjective submersions, hence so is $p$. Since $\Gamma\orb G$ is compact, it follows that $p$ is a fiber bundle.
\end{proof}

\subsection{Maps of special solvmanifolds}

Now suppose $\Gamma$ and $\Gamma'$ are two strongly torsion-free $S$-groups, with $\Gamma_1$, $\Gamma'_1$, $\Lambda_0$ and $\Lambda'_0$ as before. Then every group morphism $\varphi:\Gamma\to \Gamma'$ satisfies $\varphi(\Gamma_1)\subseteq \Gamma'_1$. In particular, $\varphi$ induces well-defined morphisms $\varphi_0:\Lambda_0\to\Lambda'_0$ and $\varphi':\Gamma_1\to \Gamma'_1$. The morphism $\varphi':\Gamma_1\to \Gamma'_1$ of finitely generated torsion-free nilpotent groups in turn has a linearisation $\{\varphi_i:\Lambda_i\to \Lambda'_i \}_{i\geq 1}$, as defined in section \ref{subsec:nilmanifolds}. We call the total collection $\{\varphi_i:\Lambda_i\to \Lambda'_i \}_{i\geq 0}$ the \emph{linearisation of $\varphi$}.

\begin{lem}\label{lem:phi-Ai}
Let $\Gamma$ and $\Gamma'$ be strongly torsion-free $S$-groups with respective linearisations $\{\Lambda_i,\mu_i \}_{i\geq 1}$ and $\{\Lambda'_i,\mu'_i \}_{i\geq 1}$ (see section \ref{subsec:solvmanifolds}), and $\varphi:\Gamma\to \Gamma'$ a morphism with linearisation $\{\varphi_i:\Lambda_i\to \Lambda'_i\}_{i\geq 0}$. Then \[
\varphi_i\mu_i(\lambda)=\mu'_i(\varphi_0(\lambda))\varphi_i
\]
for all $\lambda\in \Lambda_0$ and $i\geq 1$.
\end{lem}

\begin{proof}
Write $\lambda=\gamma\Gamma_1$ with $\gamma\in \Gamma$. Then $\varphi_0(\lambda)=\varphi(\gamma)\Gamma'_1$. So, for any element $\gamma_i\Gamma_{i+1}\in \Lambda_i$, \begin{align*}
\varphi_i(\mu_i(\lambda)(\gamma_i\Gamma_{i+1})) &= \varphi_i(\gamma\gamma_i \gamma^{-1}\Gamma_{i+1}) \\
&= \varphi(\gamma\gamma_i\gamma^{-1})\Gamma'_{i+1} \\
&= \varphi(\gamma)\varphi(\gamma_i) \varphi(\gamma)^{-1}\Gamma'_{i+1} \\
&= \mu'_i(\varphi_0(\lambda))(\varphi(\gamma_i)\Gamma'_{i+1}) \\
&= \mu'_i(\varphi_0(\lambda))(\varphi_i(\gamma_i\Gamma_{i+1})). \qedhere
\end{align*}
\end{proof}

\begin{lem}\label{lem:lin-f-sol-base}
Given two special solvmanifolds $\Gamma\orb G$ and $\Gamma'\orb G'$, suppose $f:\Gamma\orb G\to \Gamma'\orb G'$ is a fiber map of the minimal Mostow fibrations, with induced morphism $\varphi:\Gamma\to \Gamma'$, and let $\varphi_0:\Lambda_0\to \Lambda'_0$ be the morphism induced by $\varphi$. If $\bar{f}:\Gamma N\orb G\to \Gamma'N'\orb G'$ is the map of the base spaces induced by $f$, we have a commutative diagram \[
\begin{tikzcd}
\pi_1(\Gamma N\orb G) \ar[r,"\bar{f}_\#"] \ar[d] & \pi_1(\Gamma'N'\orb G') \ar[d] \\
\Lambda_0 \ar[r,"\varphi_0"] & \Lambda'_0
\end{tikzcd}
\]
where the vertical arrows are isomorphisms.
\end{lem}

\begin{proof}
Let $\tilde{f}:G\to G'$ be the lift for $f$ inducing the morphism $\varphi$. Note that we can write $\bar{f}(\Gamma Ng)=\Gamma'N'\tilde{f}(g)$ for all $g\in G$.

Let $\bar{f}'$ be the map corresponding to $\bar{f}$ under the homeomorphisms from Step \ref{step:base-space}, \[
\begin{tikzcd}
\Gamma N\orb G \ar[r,"\bar{f}"] \ar[d,"h"'] & \Gamma'N'\orb G' \ar[d,"h'"] \\
\frac{\Gamma N}{N} \orb (N\orb G) \ar[r,"\bar{f}'"] & \frac{\Gamma' N'}{N'} \orb (N'\orb G').
\end{tikzcd}
\]
Then a lift for $\bar{f}'$ is \[
\bar{\tilde{f}}:N\orb G\to N'\orb G':Ng\mapsto N'\tilde{f}(g).
\]
For all $\gamma\in \Gamma$ and $g\in G$, we have \begin{align*}
\bar{\tilde{f}}((N\gamma)(Ng))&=\bar{\tilde{f}}(N\gamma g) \\
&=N'\tilde{f}(\gamma g) \\
&=N'\varphi(\gamma)\tilde{f}(g) \\
&=(N'\varphi(\gamma))(N'\tilde{f}(g)) \\
&=(N'\varphi(\gamma))\bar{\tilde{f}}(Ng)
\end{align*}
so the morphism $N\orb \Gamma N\to N'\orb \Gamma'N'$ induced by this lift is given by $N\gamma\mapsto N'\varphi(\gamma)$. Under the identifications $N\orb \Gamma N\cong \Gamma/\Gamma_1=\Lambda_0$ and $N'\orb \Gamma' N'\cong \Gamma'/\Gamma'_1=\Lambda'_0$, this is precisely the morphism $\varphi_0:\Lambda_0\to \Lambda'_0$ induced by $\varphi$.

Together, for each $\Gamma Ng\in \Gamma N\orb G$, we get the diagram \[
\begin{tikzcd}
\pi_1(\Gamma N\orb G,\,\Gamma Ng) \ar[r,"\bar{f}_\#"] \ar[d,"h_\#"'] & \pi_1(\Gamma'N'\orb G',\,\Gamma'N'\tilde{f}(g)) \ar[d,"h'_\#"] \\
\pi_1(\frac{\Gamma N}{N}\orb (N\orb G),\frac{\Gamma N}{N}Ng) \ar[r,"\bar{f}'_\#"] \ar[d,"\Psi_{Ng}"'] & \pi_1(\frac{\Gamma'N'}{N'}\orb (N'\orb G'),\frac{\Gamma'N'}{N'}\bar{\tilde{f}}(Ng)) \ar[d,"\Psi_{\bar{\tilde{f}}(Ng)}"] \\
\frac{\Gamma N}{N} \ar[r,"{N\gamma\,\mapsto\,N'\varphi(\gamma)}"] \ar[d] & \frac{\Gamma'N'}{N'} \ar[d] \\
\Lambda_0 \ar[r,"\varphi_0"] & \Lambda'_0
\end{tikzcd}
\]
where $\Psi_{Ng}$ and $\Psi_{\bar{\tilde{f}}(Ng)}$ are the isomorphisms from section \ref{subsec:Nielsen-coin}.
\end{proof}

In Step 4 of the proof of Theorem~\ref{thm:solv-fib} we showed that for the minimal Mostow fibration $p: \Gamma \orb G \to \Gamma N \orb G$ of a special solvmanifold $\Gamma\orb G$ there is, for each $g\in G$,  a homeomorphism of each fiber to $\Gamma_1\orb N$. This is explicitly given as the map   
\[ h_g : \Gamma\orb \Gamma Ng\to \Gamma_1\orb N\]
sending an element of $\Gamma\orb \Gamma Ng$ that is represented as $\Gamma n g$ to $\Gamma_1 n \in \Gamma_1\orb N$. We now use this homeomorphism,  and the analogous homeomorphism $h'_{g'}$ for a second special solvmanifold $\Gamma' \orb G'$,  in the following lemma. 
\begin{lem}\label{lem:lin-f-sol-fiber}
Given two special solvmanifolds $\Gamma\orb G$ and $\Gamma'\orb G'$, suppose $f:\Gamma\orb G\to \Gamma'\orb G'$ is a fiber map of the minimal Mostow fibrations, with induced morphism $\varphi:\Gamma\to \Gamma'$. For $g\in G$ and for all $g'\in G'$ with $\bar{f}(\Gamma Ng)=\Gamma'N'g'$, if $f_{\Gamma Ng}:\Gamma \orb \Gamma Ng\to \Gamma'\orb \Gamma'N'g'$ is the map on the fibers induced by $f$, and $f'_{\Gamma Ng}$ is the map making the following diagram commutative (where $h_g$ and $h'_{g'}$ are as above), \[
\begin{tikzcd}
\Gamma \orb \Gamma Ng \ar[r,"f_{\Gamma Ng}"] \ar[d,"h_g"'] & \Gamma'\orb \Gamma'N'g' \ar[d,"h'_{g'}"] \\
\Gamma_1\orb N \ar[r,"f'_{\Gamma Ng}"] & \Gamma'_1\orb N',
\end{tikzcd}
\]
then the morphism induced by $f'_{\Gamma Ng}$ is $\mu(\gamma')\varphi'$ for some $\gamma'\in \Gamma'$, where $\varphi':\Gamma_1\to \Gamma'_1$ is the morphism induced by $\varphi$.
\end{lem}


\begin{proof}
Let $\tilde{f}:G\to G'$ be the lift inducing the morphism $\varphi$. Since $\Gamma'\tilde{f}(g)=f(\Gamma g)\in \Gamma'N'g'$, there is a $\gamma'\in \Gamma'$ such that $\gamma'\tilde{f}(g)\in N'g'$. We claim that $\gamma'\tilde{f}(Ng)\subseteq N'g'$. 
Assuming the claim is true, the map \[
\tilde{f}':N\to N':n\mapsto \gamma'\tilde{f}(ng)(g')^{-1}
\]
is well-defined and continuous, and one easily checks that this is a lift for $f'_{\Gamma Ng}$. 
We now show that the morphism induced by this lift is $\mu(\gamma')\varphi'$. Indeed, for all $\gamma_1\in \Gamma_1$ and $n\in N$, \begin{align*}
\tilde{f}'(\gamma_1n)&=\gamma'\tilde{f}(\gamma_1ng)(g')^{-1}\\ 
&=\gamma'\varphi(\gamma_1)\tilde{f}(ng)(g')^{-1} \\
&=\gamma'\varphi(\gamma_1)(\gamma')^{-1}\tilde{f}'(n).
\end{align*}
               
Finally, we prove the claim. To show that $\gamma'\tilde{f}(Ng)\subseteq N'g'$, note that $Ng$ is path connected, so $\gamma'\tilde{f}(Ng)$ is a path connected subset of $\gamma'\tilde{f}(\Gamma Ng)\subseteq \Gamma'N'g'$ containing $\gamma'\tilde{f}(g)\vcentcolon=g''$. We show that the path component of $\Gamma'N'g'$ containing $g''$ is contained in $N'g'$.

Suppose $\alpha:[0,1]\to \Gamma'N'g'$ is a path starting at $g''$. Consider the map $\bar{\alpha}:[0,1]\to N'\orb G'$ induced by $\alpha$. Then $\bar{\alpha}([0,1])\subseteq N'\orb \Gamma'N'g'$ and $\bar{\alpha}(0)=N'g''=N'g'$. Note that $\Gamma'$ is countable, and the preimages $\bar{\alpha}^{-1}(\{\gamma'' N'g'\})$, $\gamma''\in \Gamma'$, are disjoint closed sets whose union is $[0,1]$. Since $[0,1]$ cannot be written as a non-trivial countable disjoint union of closed sets, we must have $\bar{\alpha}([0,1])=\{N'g'\}$, so $\alpha([0,1])\subseteq N'g'$.
\end{proof}

\section{Nielsen coincidence numbers on special solvmanifolds}

By a result of Jezierski \cite{jezierski1989}, which we will adapt to our purposes, the Nielsen coincidence number of fiber maps can be expressed in terms of the Nielsen numbers of the induced maps on the base spaces and fibers. By combining this result with Lemma \ref{lem:lin-f-nil} and Lemmas \ref{lem:lin-f-sol-base} and \ref{lem:lin-f-sol-fiber}, we can express Nielsen numbers on special solvmanifolds entirely in terms of Nielsen numbers on tori, for which a general formula is known.

\subsection{Coincidences on tori}

The following is a standard result, shown e.g.\ in \cite[Lemma 7.3]{jezierski1989}:

\begin{thm}\label{thm:N-tori}
Let $T$ and $T'$ be tori of equal dimension $k$. For two maps $f,g:T\to T'$ with respective induced morphisms $f_\#,g_\#:\pi_1(T)\cong\ZZ^k\to\pi_1(T')\cong\ZZ^k$, \[
N(f,g)=\left|\det(g_\#-f_\#)\right|.
\]
\end{thm}

\begin{rmk}\label{rmk:det-indep}
Note that this formula is independent of the chosen identifications $\pi_1(T)\cong\ZZ^k$ and $\pi_1(T')\cong\ZZ^k$: say $F,G\in \ZZ^{k\times k}$ are the respective matrices of $f_\#$ and $g_\#$ for one choice of bases for $\pi_1(T)$ and $\pi_1(T')$, and $F',G'\in \ZZ^{k\times k}$ are those for another choice. The change of basis from the first to the second basis for $\pi_1(T)$ is given by a matrix $P\in \ZZ^{k\times k}$ with $\left|\det(P)\right|=1$; the one for $\pi_1(T')$ by another matrix $P'\in \ZZ^{k\times k}$ with $\left|\det(P')\right|=1$. Thus, \begin{align*}
\left|\det(G'-F')\right|&=|\!\det(P'GP^{-1}-P'FP^{-1})|\\&=\left|\det(P')\right|\left|\det(G-F)\right||\!\det(P^{-1})|\\&=\left|\det(G-F)\right|.
\end{align*}
\end{rmk}

\subsection{Coincidences of fiber maps}

Given a commutative diagram of fiber bundles and fiber maps \[
\begin{tikzcd}
E \ar[r,"{(f,g)}"] \ar[d,"p"'] & E' \ar[d,"p'"] \\
B \ar[r,"{(\bar{f},\bar{g})}"] & B'
\end{tikzcd}
\]
where $E,E',B,B'$ and all fibers are compact connected orientable manifolds with $\dim E=\dim E'$ and $\dim B=\dim B'$, we can consider the Nielsen numbers $N(f,g)$, $N(\bar{f},\bar{g})$, and $N(f_b,g_b)$ for all $b\in \Coin(\bar{f},\bar{g})$ (where $f_b$ and $g_b$ are the maps $E_b\vcentcolon=p^{-1}(b)\to {E'}_{\!b'}\vcentcolon=(p')^{-1}(b')$ induced by $f$ and $g$, with $b'=\bar{f}(b)=\bar{g}(b)$).

The main focus of \cite{jezierski1989} lies on proving the product formula 
\[
N(f,g)=N(\bar{f},\bar{g})N(f_b,g_b)
\]
for a pair of fiber maps between orientable fiber bundles (under some extra conditions; see \cite[Theorem 6.5]{jezierski1989}). Our goal is different: we want a formula for $N(f,g)$ in terms of $N(\bar{f},\bar{g})$ and $N(f_b,g_b)$ that can be applied to the Mostow fibrations of solvmanifolds, which are non-orientable fiber bundles (see \cite[Lemma 1.3]{mccord1991}), but which are simple in the sense that the base spaces are tori. We will prove:

\begin{thm}\label{thm:N-fiber-tori}
Given a commutative diagram of fiber bundles and fiber maps \[
\begin{tikzcd}
E \ar[r,"{(f,g)}"] \ar[d,"p"'] & E' \ar[d,"p'"] \\
B \ar[r,"{(\bar{f},\bar{g})}"] & B'
\end{tikzcd}
\]
where all spaces involved are compact connected orientable manifolds of respectively equal dimensions, and the base spaces $B$ and $B'$ are tori. Then \[
N(f,g)=\sum_b N(f_b,g_b)
\]
where the sum runs over representatives for the essential coincidence classes of $(\bar{f},\bar{g})$.
\end{thm}

We will use the machinery developed in \cite{jezierski1989}, but most of this will become trivial in case the base spaces are tori, so we will not go over all the definitions in general and refer the reader to \cite{jezierski1989} for more details.

The results of \cite{jezierski1989} that we recall below are valid for any commutative diagram \[
\begin{tikzcd}
E \ar[r,"{(f,g)}"] \ar[d,"p"'] & E' \ar[d,"p'"] \\
B \ar[r,"{(\bar{f},\bar{g})}"] & B'
\end{tikzcd}
\]
of fiber bundles where all spaces involved are compact connected orientable manifolds of respectively equal dimensions.

For $e'\in E'$, we can consider the group \[
K=\ker(\pi_1({E'}_{\!\!p'(e')},e')\to \pi_1(E',e')).
\]
Since all fibers of a fiber bundle are homeomorphic and the spaces involved are path connected, this is the same group for all $e'\in E'$.

In \cite[Section 1]{jezierski1989}, the sets of \emph{$K$-Nielsen classes} $\nabla_{\!K}(f_b,g_b)$ are defined. We will only need this in the special case $K=1$, where these coincide with the ordinary coincidence classes of the pair $(f_b,g_b)$.

At the beginning of \cite[Section 6]{jezierski1989} it is observed that, if $\bar{A}_1,\ldots,\bar{A}_s$ are the essential coincidence classes of $(\bar{f},\bar{g})$, then \begin{equation}\label{eq:N-Ci}
N(f,g)=C_1+\ldots+C_s
\end{equation}
where, for any $b\in \bar{A}_i$, the number $C_i$ is equal to the number of \emph{essential orbits} (i.e.\ the orbits that contain essential classes) of the action of \[
C(\bar{f}_\#,\bar{g}_\#)=\{\alpha\in \pi_1(B)\mid \bar{f}_\#(\alpha)=\bar{g}_\#(\alpha) \}
\]
on $\nabla_{\!K}(f_b,g_b)$ \cite[Lemma 6.1]{jezierski1989}. Again, we will not go over the definition of this action as we will only need this in case $C(\bar{f}_\#,\bar{g}_\#)=1$. In that case, all orbits are singletons, and $C_i$ is just the number of essential $K$-Nielsen classes of $(f_b,g_b)$. If moreover $K=1$, this is just the number of essential coincidence classes, i.e.\ $N(f_b,g_b)$.

We can conclude that, if $K=1$ and $C(\bar{f}_\#,\bar{g}_\#)=1$, then \[
N(f,g)=\sum_b N(f_b,g_b)
\]
where the sum runs over representatives for the essential coincidence classes of $(\bar{f},\bar{g})$. Note that in case $N(\bar{f},\bar{g})=0$, this formula holds regardless of the conditions $K=1$ and $C(\bar{f}_\#,\bar{g}_\#)=1$: since there are no essential coincidence classes of $(\bar{f},\bar{g})$, the sum on the right is empty, and equation \eqref{eq:N-Ci} implies that $N(f,g)=0$.

Thus, to prove Theorem \ref{thm:N-fiber-tori}, it suffices to prove that if the base spaces are tori and $N(\bar{f},\bar{g})\neq 0$, then $K=1$ and $C(\bar{f}_\#,\bar{g}_\#)=1$.

\begin{proof}[Proof of Theorem \ref{thm:N-fiber-tori}]
Suppose $B$ and $B'$ are tori of equal dimension $k$. 

Choose $e'\in E'$. Since $\pi_2(B')=1$, it follows from the long exact homotopy sequence \[
\cdots \longrightarrow \pi_2(B',p'(e')) \longrightarrow \pi_1({E'}_{\!\!p'(e')},e') \longrightarrow \pi_1(E',e') \longrightarrow \cdots
\]
that the map $\pi_1({E'}_{\!\!p'(e')},e') \to \pi_1(E',e')$ is injective, hence $K=1$. 

Now suppose $N(\bar{f},\bar{g})\neq 0$; we will show that $C(\bar{f}_\#,\bar{g}_\#)=1$.
Since $B$ and $B'$ are $k$-dimensional tori, both are homeomorphic to $\ZZ^k\orb \RR^k$. There are isomorphisms $\pi_1(B),\pi_1(B')\cong \ZZ^k$ under which $\bar{f}_\#,\bar{g}_\#:\ZZ^k\to \ZZ^k$ coincide with the induced morphisms of $f$ and $g$, so that Theorem \ref{thm:N-tori} yields $N(\bar{f},\bar{g})=\left|\det(\bar{g}_\#-\bar{f}_\#)\right|$. Under those isomorphisms, an element of $C(\bar{f}_\#,\bar{g}_\#)$ can be viewed as an element $v\in \ZZ^k$ with $(\bar{f}_\#-\bar{g}_\#)v=0$. Since $N(\bar{f},\bar{g})\neq 0$, so $\det(\bar{f}_\#-\bar{g}_\#)\neq 0$, the only such element is $v=0$.
\end{proof}

\subsection{Coincidences on solvmanifolds}

We now consider coincidences of two maps $f,q:\Gamma\orb G\to\Gamma'\orb G'$ between special solvmanifolds, where we will later specialise to the case where $q$ is a finite covering map.

\begin{thm}\label{thm:N-sol}
Let $f,q:\Gamma\orb G\to\Gamma'\orb G'$ be maps between equal-dimensional special solvmanifolds, with induced morphisms $\varphi,\psi:\Gamma\to \Gamma'$. Let $\{\varphi_i:\Lambda_i\to \Lambda'_i \}_{i\geq 0}$ and $\{\psi_i:\Lambda_i\to \Lambda'_i \}_{i\geq 0}$ be the respective linearisations of $\varphi$ and $\psi$. Suppose the following two conditions are satisfied: \begin{itemize}
\item[(1)] For all $i\geq 0$, the groups $\Lambda_i$ and $\Lambda'_i$ have the same rank $k_i$. In particular, we can view $\varphi_i$ and $\psi_i$ as linear maps $\ZZ^{k_i}\to \ZZ^{k_i}$, and the nilpotency class of $\Gamma_1$ and $\Gamma'_1$ is the same, say $c$.
\item[(2)] Either $\det(\psi_0-\varphi_0)=0$, or the expression $\det(\psi_i-\mu'_i(v)\varphi_i)$ is independent of $v\in \ZZ^{k_0}$ for all $i\geq 1$, where $\{\mu'_i:\ZZ^{k_0}\to \GL_{k_i}(\ZZ)\}_{i\geq 1}$ is the linearisation of $\Gamma'$.
\end{itemize}
Then \[
N(f,q)=\prod_{i=0}^c \left|\det(\psi_i-\varphi_i)\right|.
\]
\end{thm}

\begin{proof}
Let $p:\Gamma\orb G\to \Gamma N\orb G$ and $p':\Gamma'\orb G'\to \Gamma' N'\orb G'$ be the minimal Mostow fibrations. By condition (1), the corresponding spaces in these fibrations have respectively equal dimensions.

Since the Nielsen number is a homotopy invariant, we may assume $f$ and $q$ are fiber maps of the fibrations $p$ and $p'$. These maps still induce the same morphisms $\varphi$ and $\psi$. By Theorem \ref{thm:N-fiber-tori}, we can write \[
N(f,q)=\sum_{\Gamma Ng} N(f_{\Gamma Ng},q_{\Gamma Ng})
\]
where the sum runs over representatives for the essential coincidence classes of the pair $(\bar{f},\bar{q})$.

By Lemma \ref{lem:lin-f-sol-base}, the maps $\bar{f}$ and $\bar{q}$ are maps on tori with respective fundamental group morphisms $\varphi_0:\ZZ^{k_0}\to \ZZ^{k_0}$ and $\psi_0:\ZZ^{k_0}\to \ZZ^{k_0}$. Thus, by Theorem \ref{thm:N-tori}, \[
N(\bar{f},\bar{q})=\left|\det(\psi_0-\varphi_0)\right|.
\]
If this is zero, the desired equality for $N(f,g)$ holds. From now on, we suppose this is non-zero.
In that case, consider a point $\Gamma Ng\in \Coin(\bar{f},\bar{q})$, and choose $g'\in G'$ such that $\bar{f}(\Gamma Ng)=\bar{q}(\Gamma Ng)=\Gamma'N'g'$. By Lemma \ref{lem:lin-f-sol-fiber}, the maps $f'_{\Gamma Ng}$ and $q'_{\Gamma Ng}$ fitting into the diagram \[
\begin{tikzcd}
\Gamma \orb \Gamma Ng \ar[r,"f_{\Gamma Ng}","q_{\Gamma Ng}"'] \ar[d,"h_g"'] & \Gamma'\orb \Gamma'N'g' \ar[d,"h_{g'}"] \\
\Gamma_1\orb N \ar[r,"f'_{\Gamma Ng}","q'_{\Gamma Ng}"'] & \Gamma'_1\orb N'
\end{tikzcd}
\]
respectively induce the morphisms $\mu(\gamma_g)\varphi':\Gamma_1\to \Gamma'_1$ and $\mu(\gamma'_g)\psi':\Gamma_1\to \Gamma'_1$, for some $\gamma_g,\gamma'_g\in \Gamma'$ depending on $\Gamma Ng$. Since $h_g$ and $h_{g'}$ are homeomorphisms, we can write \[
N(f_{\Gamma Ng},q_{\Gamma Ng})=N(h_{g'}^{-1}f'_{\Gamma Ng}h_g,h_{g'}^{-1}q'_{\Gamma Ng}h_g)=N(f'_{\Gamma Ng},q'_{\Gamma Ng}).
\]

Let $p_1:\Gamma_1\orb N_1\to \Gamma_1N_2\orb N_1$ and $p'_1:\Gamma'_1\orb N'_1\to \Gamma'_1N'_2\orb N'_1$ be the fibrations from Lemma \ref{lem:nil-fib}. By condition (1), the corresponding spaces in these fibrations have respectively equal dimensions. We can apply the same reasoning as above (now using Lemma \ref{lem:lin-f-nil} instead of Lemmas \ref{lem:lin-f-sol-base} and \ref{lem:lin-f-sol-fiber}) to the maps $f'_{\Gamma Ng},q'_{\Gamma Ng}:\Gamma_1\orb N_1\to\Gamma'_1\orb N'_1$. Note that the morphisms $\mu(\gamma_g)\varphi'$ and $\mu(\gamma'_g)\psi'$ induced by $f'_{\Gamma Ng}$ and $q'_{\Gamma Ng}$ have respective linearisations $\{\mu'_i(v_g)\varphi_i:\ZZ^{k_i}\to \ZZ^{k_i}\}_{i\geq 1}$ and $\{\mu'_i(v'_g)\psi_i:\ZZ^{k_i}\to \ZZ^{k_i}\}_{i\geq 1}$, where $v_g,v'_g\in \ZZ^{k_0}$ are the elements representing $\gamma_g\Gamma'_1,\gamma'_g\Gamma'_1\in \Lambda'_0$.

We get \[
N(f'_{\Gamma Ng},g'_{\Gamma Ng})=\sum_{\Gamma_1N_2n} N((f'_{\Gamma Ng})'_{\Gamma_1N_2n},(\bar{q}'_{\Gamma Ng})'_{\Gamma_1N_2n})
\]
where the sum runs over representatives for the essential coincidence classes of $(\bar{f}'_{\Gamma Ng},\bar{q}'_{\Gamma Ng})$, and therefore has \[
N(\bar{f}'_{\Gamma Ng},\bar{q}'_{\Gamma Ng})=\left|\det(\mu'_1(v_g)\varphi_1-\mu'_1(v'_g)\psi_1)\right|
\]
terms. The respective morphisms $\Gamma_2\to \Gamma'_2$ induced by $(\bar{f}'_{\Gamma Ng})'_{\Gamma_1N_2n}$ and $(\bar{q}'_{\Gamma Ng})'_{\Gamma_1N_2n}$ are $\mu(\gamma_n)\mu(\gamma_g)\varphi':\Gamma_2\to \Gamma'_2$ and $\mu(\gamma'_n)\mu(\gamma'_g)\psi':\Gamma_2\to \Gamma'_2$, for some $\gamma_n,\gamma'_n\in \Gamma'_1$ depending on $\Gamma_1N_2n$. But for all $\gamma_1\in \Gamma'_1$ and $i\geq 2$, the morphism \[ 
\Lambda'_i\to \Lambda'_i:\gamma_i\Gamma'_{i+1}\mapsto \gamma_1\gamma_i(\gamma_1)^{-1}\Gamma'_{i+1}=[\gamma_1,\gamma_i]\gamma_i\Gamma'_{i+1}, 
\]
induced by $\mu(\gamma_1)$ is trivial, since $[\Gamma'_1,\Gamma'_i]\subseteq \Gamma'_{i+1}$. Thus, the morphisms $\mu(\gamma_n)\mu(\gamma_g)\varphi'$ and $\mu(\gamma'_n)\mu(\gamma'_g)\psi'$ induced by $(\bar{f}'_{\Gamma Ng})'_{\Gamma_1N_2n}$ and $(\bar{q}'_{\Gamma Ng})'_{\Gamma_1N_2n}$ have linearisations $\{\mu'_i(v_g)\varphi_i:\ZZ^{k_i}\to \ZZ^{k_i}\}_{i\geq 2}$ and $\{\mu'_i(v'_g)\psi_i:\ZZ^{k_i}\to \ZZ^{k_i}\}_{i\geq 2}$, independently of $n$. 

Continuing inductively, we find \[
N(f,g)=\sum_{\Gamma Ng} \prod_{i=1}^c \left|\det(\mu'_i(v'_g)\psi_i-\mu'_i(v_g)\varphi_i)\right|
\]
where the sum runs over representatives for the essential coincidence classes of $(\bar{f},\bar{g})$ and has $N(\bar{f},\bar{g})=\left|\det(\psi_0-\varphi_0)\right|$ terms. Note that \begin{align*}
\left|\det(\mu'_i(v'_g)\psi_i-\mu'_i(v_g)\varphi_i)\right|&=\left|\det(\mu'_i(v'_g))\right|\left|\det(\psi_i-\mu'_i(v_g-v'_g)\varphi_i)\right|\\&=\left|\det(\psi_i-\mu'_i(v_g-v'_g)\varphi_i)\right|,
\end{align*}
so under condition (2) (keeping in mind that we were assuming $N(\bar{f},\bar{g})\neq 0$), we obtain \[
N(f,g)=\left|\det(\psi_0-\varphi_0)\right|\prod_{i=1}^c \left|\det(\psi_i-\varphi_i)\right|=\prod_{i=0}^c \left|\det(\psi_i-\varphi_i)\right|. \qedhere
\]
\end{proof}

\begin{rmk}\label{rmk:N-sol-indep}
The induced morphisms $\varphi$ and $\psi$ are only defined up to conjugation with an element in $\Gamma'$ (determined by the choice of lifts of $f$ and $q$ used to compute them). However, the formula is independent of this choice: if $\mu(\gamma)\varphi$ and $\mu(\gamma')\psi$ are other morphisms induced by $f$ and $q$ (with $\gamma,\gamma'\in \Gamma'$), their linearisation morphisms for $i=0$ are also $\varphi_0$ and $\psi_0$; for $i\geq 1$ they are $\mu'_i(v)\varphi_i$ and $\mu'_i(v')\psi_i$, where $v,v'\in \ZZ^{k_0}$ are the elements representing $\gamma\Gamma'_1,\gamma'\Gamma'_1\in \Lambda'_0$, for $j=1,2$. If $\det(\psi_0-\varphi_0)=0$, the formula remains the same. Otherwise, it follows from condition (2) that
\begin{align*}
\left|\det(\mu'_i(v')\psi_i-\mu'_i(v)\varphi_i)\right|&=\left|\det(\mu'_i(v'))\right|\left|\det(\psi_i-\mu'_i(v-v')\varphi_i)\right|\\&=\left|\det(\psi_i-\varphi_i)\right|
\end{align*}
for all $i\geq 1$.
\end{rmk}

\begin{ex}
\label{ex:first-formula} Consider the matrix $A=\begin{bmatrix}2 & 1 \\ 1 & 1\end{bmatrix}$. This matrix has eigenvalues $\alpha=\frac12(3 + \sqrt5)$ and $\alpha^{-1}$, hence we can consider a matrix $P\in \GL_2(\RR)$ with 
\[ A = P \begin{bmatrix} \alpha & 0 \\ 0 & \alpha^{-1} \end{bmatrix}P^{-1}\]
and define
\[ A^t = P \begin{bmatrix} \alpha^t & 0 \\ 0 & \alpha^{-t} \end{bmatrix}P^{-1} ,\quad \forall t \in \RR.\]
Let $G=G'= \RR^2 \rtimes_\eta \RR$ with $\eta:\RR \to \GL_2(\RR): t \mapsto A^t$, and take $\Gamma' = \ZZ^2 \rtimes_\eta \ZZ$ and $\Gamma= \ZZ^2 \rtimes_\eta (2 \ZZ)$.

We consider the coincidences of the map $f: \Gamma \orb G \to \Gamma' \orb G : \Gamma (\bar{x}, y) \mapsto \Gamma' ( \bar{0},- \frac{y}{2})$ with the natural projection $q: \Gamma \orb G \to \Gamma' \orb G : \Gamma (\bar{x}, y) \mapsto \Gamma' (\bar{x}, y )$.
To see that $f$ is well-defined, consider the map $\tilde{f}: G \to G : (\bar{x} , y) \mapsto (\bar{0}, -\frac{y}{2})$. Note that for all $(\bar{z}, 2 t ) \in \Gamma$ (so $\bar{z} 
\in \ZZ^2$ and $t\in \ZZ$) it holds that 
\[ 
\tilde{f}((\bar{z}, 2 t ) (\bar{x}, y) )=\tilde{f}(\bar{z}+A^{2t}\bar{x},2t+y)= (\bar{0}, - t -\tfrac{y}{2}) = (\bar{0}, -t) \tilde{f}(\bar{x}, y), \]
from which we indeed get that $f$ is well-defined and that the morphism induced by $f$ is given by 
\[ \varphi: \Gamma \to \Gamma' : (\bar{z}, 2 t )\mapsto (\bar{0}, -t).\]
As a lift $\tilde{q}$ of $q$ we can just take the identity map on $G$, so the morphism induced by $q$ is the inclusion $\psi:\Gamma \to \Gamma': (\bar{z}, 2 t ) \mapsto (\bar{z}, 2 t )$. 

Note that $\Lambda_0 =2 \ZZ$ and $\Lambda_0' = \ZZ$. In order to describe $\varphi_0$ and $\psi_0$ as in Theorem~\ref{thm:N-sol} we have to identify both $\Lambda_0$ and $\Lambda_0'$ with $\ZZ$.
It clear that under $\varphi_0$ the generator of $\Lambda_0$ (which is $2$) is mapped to 
$-1$, so $\varphi_0: \ZZ \to \ZZ: 1 \mapsto -1$ or $\varphi_0$ can be identified with the $1\times 1$--matrix $\begin{bmatrix} -1 \end{bmatrix}$. For $\psi_0$, which is just the inclusion, the generator is mapped onto 2 times the generator, so $\psi_0$ can be identified with the $1\times 1$--matrix $\begin{bmatrix} 2 \end{bmatrix}$.
It follows that $\det(\psi_0 - \varphi_0) = 2 - (-1) =3 \neq 0$.

We also have that $\Lambda_1 = \Lambda_1' = \ZZ^2$, $\varphi_1 = \begin{bmatrix} 0 & 0 \\ 0 & 0 
\end{bmatrix} $ and $\psi_1= \begin{bmatrix} 1 & 0 \\ 0 & 1 \end{bmatrix}$.
As $\varphi_1$ is the zero matrix, the condition (2) of Theorem~\ref{thm:N-sol} is trivially satisfied and we find that 
\[ N(f,q) = \left|\det(\psi_0 - \varphi_0)\right| \cdot \left|\det(\psi_1 - \varphi_1)\right| = 3 .\]
In fact, one can check that $\Coin(f,q)=\{ \Gamma(\bar0,0),\; \Gamma(\bar0,\frac23),\; \Gamma(\bar0,\frac43)\}$ and so $f$ and $q$ realise the minimum number of coincidence points among all pairs of maps that are homotopic to $(f,q)$.
\end{ex}

From now on, we specialise to the setting where $G=G'$, $\Gamma'$ is net (see Definition \ref{def:net}), $\Gamma$ is a finite index index subgroup of $\Gamma'$ and $q:\Gamma\orb G\to \Gamma'\orb G$ is the natural projection. In this setting, we will show that the conditions (1) and (2) are satisfied.

\begin{lem}\label{lem:gamma-fi}
If $\Gamma$ is a finite index subgroup of $\Gamma'$, then $\Gamma_1$ is a finite index subgroup of $\Gamma'_1$, and $\Lambda_0$ can be identified with a finite index subgroup of $\Lambda'_0$. (In particular, they have the same rank.)
\end{lem}

\begin{proof}
First, we show that $\Gamma_1=\Gamma\cap\Gamma'_1$. The inclusion $\Gamma_1\subseteq \Gamma\cap\Gamma'_1$ is clear. For the converse, we use \cite[Lemma 5.5]{iris}, which states that if $\Gamma'$ is $\NR$ (so in particular if it is net) and $\Gamma$ is a finite index subgroup of $\Gamma'$, then $[\Gamma,\Gamma]$ is a finite index subgroup of $[\Gamma',\Gamma']$. Say it is of index $r$. For $g\in \Gamma\cap\Gamma'_1$, take $\ell$ such that $g^\ell\in [\Gamma',\Gamma']$; then $g^{\ell r}\in [\Gamma,\Gamma]$, so $g\in \Gamma_1$.

Since $\Gamma$ is of finite index in $\Gamma'$, the group $\Gamma\cap\Gamma'_1$ is of finite index in $\Gamma'_1$. For the second statement, we can write \[
\Gamma/\Gamma_1=\Gamma/(\Gamma\cap\Gamma'_1)\cong \Gamma\Gamma'_1/\Gamma'_1.
\]
Since $\Gamma$, so certainly $\Gamma\Gamma'_1$, is a finite index subgroup of $\Gamma'$, this quotient is a finite index subgroup of $\Gamma'/\Gamma'_1$.
\end{proof}

\begin{lem}\label{lem:gamma-i-fi}
If $\Gamma_1$ is a finite index subgroup of $\Gamma'_1$, then $\Gamma_i$ is a finite index subgroup of $\Gamma'_i$, and $\Lambda_i$ can be identified with a finite index subgroup of $\Lambda'_i$, for all $i\geq 1$. (In particular, they have the same rank.)
\end{lem}

\begin{proof}
Let $N$ be the Mal'cev completion of $\Gamma'_1$, which is also the Mal'cev completion of $\Gamma_1$. Then $\Gamma_i=\Gamma_1\cap \gamma_i(N)$ and $\Gamma'_i=\Gamma'_1\cap \gamma_i(N)$. In particular, $\Gamma_i=\Gamma_1\cap \Gamma'_i$, and the proof proceeds in analogy with the proof of Lemma \ref{lem:gamma-fi}.
\end{proof}

\begin{rmk}
Note that the assumption that $\Gamma'$ is net is only used in Lemma \ref{lem:gamma-fi}, not in Lemma \ref{lem:gamma-i-fi}.
\end{rmk}

We will explicitly use how each group $\Lambda_i$ is identified with a finite index subgroup of $\Lambda'_i$. If we write $\Gamma=\Gamma_0$ and $\Gamma'=\Gamma'_0$, then for all $i\geq 0$, we identified \[
\Lambda_i=\Gamma_i/\Gamma_{i+1}=\Gamma_i/(\Gamma_i\cap \Gamma'_{i+1})\cong \Gamma_i\Gamma'_{i+1}/\Gamma'_{i+1} \subseteq \Gamma'_i/\Gamma'_{i+1}=\Lambda'_i.
\]
The isomorphism $\chi_i:\Gamma_i/\Gamma_{i+1}\to \Gamma_i\Gamma'_{i+1}/\Gamma'_{i+1}$ is given by $\chi_i(\gamma_i\Gamma_{i+1})=\gamma_i\Gamma'_{i+1}$.

\begin{lem}\label{lem:Ai'=Ai}
Under the above identifications, the linearisation morphisms of the inclusion $\Gamma\to \Gamma'$ are the inclusions $\Lambda_i\to \Lambda'_i$. If $\{\Lambda_i,\mu_i \}_{i\geq 1}$ and $\{\Lambda'_i,\mu'_i \}_{i\geq 1}$ are the linearisations of $\Gamma$ and $\Gamma'$, then $\mu_i(\lambda)(\lambda_i)=\mu'_i(\lambda)(\lambda_i)$ for all $\lambda\in \Lambda_0$ and $\lambda_i\in \Lambda_i$, for all $i\geq 1$.
\end{lem}

\begin{proof}
For all $i\geq 0$, the map $\Lambda_i\to \Lambda'_i$ induced by the inclusion $\iota:\Gamma\to \Gamma'$ is \[
\iota_i:\Gamma_i/\Gamma_{i+1}\to \Gamma'_i/\Gamma'_{i+1}: \gamma_i\Gamma_{i+1} \mapsto \iota(\gamma_i)\Gamma'_{i+1}=\gamma_i\Gamma'_{i+1}.
\]
Writing an element of $\Gamma_i\Gamma'_{i+1}/\Gamma'_{i+1}$ as $\gamma_i\Gamma'_{i+1}$ with $\gamma_i\in \Gamma_i$, we get \[
\iota_i\circ \chi_i^{-1}: \Gamma_i\Gamma'_{i+1}/\Gamma'_{i+1}\to \Gamma_i/\Gamma_{i+1}\to \Gamma'_i/\Gamma'_{i+1}: \gamma_i\Gamma'_{i+1}\mapsto \gamma_i\Gamma_{i+1}\mapsto \gamma_i\Gamma'_{i+1},
\]
which is indeed the inclusion $\Gamma_i\Gamma'_{i+1}/\Gamma'_{i+1}\to \Gamma'_i/\Gamma'_{i+1}$. The second statement now follows by applying Lemma \ref{lem:phi-Ai} to the inclusion $\Gamma\to \Gamma'$.
\end{proof}

Fix an identification $\Lambda'_i\cong \ZZ^{k_i}$ for all $i\geq 0$. Then $\Lambda_i$ corresponds to a subgroup $B_i\ZZ^{k_i}$ for some matrix $B_i\in \ZZ^{k_i\times k_i}$ with $\left|\det(B_i)\right|=[\Lambda'_i:\Lambda_i]$. Under these identifications, let $\mu'_i:\ZZ^{k_0}\to \GL_{k_i}(\ZZ)$ and $\mu_i:B_0\ZZ^{k_0}\to \GL(B_i\ZZ^{k_i})$ be the linearisation morphisms of $\Gamma'$ and $\Gamma$. If, for $v\in \ZZ^{k_0}$, the linear map $\mu'_i(v):\ZZ^{k_i}\to \ZZ^{k_i}$ is given by multiplication with a matrix $M_i(v)\in \ZZ^{k_i\times k_i}$, then by Lemma \ref{lem:Ai'=Ai}, for $v\in B_0\ZZ^{k_0}$, the linear map $\mu_i(v):B_i\ZZ^{k_i}\to B_i\ZZ^{k_i}$ is given by multiplication with the same matrix $M_i(v)$. Since $\Gamma'$ is net, each of the matrices $M_i(v)$ is net, i.e.\ the multiplicative subgroup of $\CC^*$ generated by the eigenvalues of $M_i(v)$ contains no non-trivial roots of unity.

Let $\{\varphi_i:\Lambda_i\to \Lambda'_i \}_{i\geq 0}$ be the linearisation of the morphism $\varphi:\Gamma\to \Gamma'$ induced by $f:\Gamma\orb G\to \Gamma'\orb G$. Under the above identifications, $\varphi_i$ corresponds to a linear map $B_i\ZZ^{k_i}\to \ZZ^{k_i}$, which is given by multiplication with a matrix $F_i\in \QQ^{k_i\times k_i}$. 

\begin{df}
We will call $\{F_i\}_{i=0}^c$ the set of \emph{linearisation matrices} of $\varphi$ (or the linearisation matrices of $f$, where we keep in mind that they depend on a choice of lift for $f$ used to compute the morphism $\varphi$).
\end{df}

By Lemma \ref{lem:phi-Ai}, these matrices satisfy \[
F_iM_i(v)=M_i(F_0v)F_i.
\]
for all $v\in B_0\ZZ^{k_0}$ and $i\geq 1$.

Finally, consider the projection $q:\Gamma\orb G\to \Gamma'\orb G:\Gamma g\mapsto \Gamma'g$. As a lift for $q$, we can take the identity map $G\to G$, whose induced morphism $\psi:\Gamma\to \Gamma'$ is the subgroup inclusion. By Lemma \ref{lem:Ai'=Ai}, its linearisation consists of the inclusion morphisms $\Lambda_i\to \Lambda'_i$; the corresponding linear maps $B_i\ZZ^{k_i}\to \ZZ^{k_i}$ are given by multiplication with the identity matrix $I\in \ZZ^{k_i\times k_i}$.

We can now use the following technical result from \cite{iris}.

\begin{lem}[{\cite[Lemma 5.7]{iris}}]\label{lem:iris}
If $X\in \CC^{n\times n}$ and $\Phi\in \QQ^{m\times m}$ are matrices, $A:\ZZ^m\to \GL_n(\ZZ)$ is a morphism, and there is a $d\in \ZZ_{>0}$ such that \begin{itemize}
\item[\rm(a)] $\Phi(d\ZZ^m)\subseteq \ZZ^m$;
\item[\rm(b)] $XA(dv)=A(\Phi (dv))X$ for all $v\in \ZZ^m$;
\item[\rm(c)] $\Phi$ has no eigenvalue $1$;
\item[\rm(d)] $A(v)$ is net for all $v\in \ZZ^m$;
\end{itemize}
then $\det(I-A(v)X)$ is independent of $v\in \ZZ^m$.
\end{lem}

\begin{rmk}
In the statement of \cite[Lemma 5.7]{iris}, the matrix $X$ has integer entries, but the proof only assumes $X\in \CC^{n\times n}$.
\end{rmk}

To apply this to our setting, with $F_i\in \CC^{k_i\times k_i}$ for $X$, $F_0\in \QQ^{k_0\times k_0}$ for $\Phi$ and $\mu'_i:\ZZ^{k_0}\to \GL_{k_i}(\ZZ):v\mapsto M_i(v)$ for $A$, take $d\in \ZZ_{>0}$ such that $d\ZZ^{k_0}\subseteq B_0\ZZ^{k_0}$. Then certainly $F_0(d\ZZ^{k_0})\subseteq\ZZ^{k_0}$ and $F_iM_i(v)=M_i(F_0v)F_i$ for all $v\in d\ZZ^{k_0}$. Thus, if $F_0$ has no eigenvalue $1$, then $\det(I-M_i(v)F_i)$ is independent of $v\in \ZZ^{k_0}$. 

\medskip

To apply Theorem \ref{thm:N-sol}, we need to identify both $\Lambda_i$ and $\Lambda'_i$ with $\ZZ^{k_i}$. We can keep the above identification for $\Lambda'_i$, and for $\Lambda_i\cong B_i\ZZ^{k_i}$ consider the isomorphism \[
\ZZ^{k_i}\to B_i\ZZ^{k_i}:v\mapsto B_iv.
\]
Under this isomorphism, the morphisms $\varphi_i:\ZZ^{k_i}\to \ZZ^{k_i}$ are given by multiplication with $F_iB_i$, and the morphisms $\psi_i:\ZZ^{k_i}\to \ZZ^{k_i}$ are given by multiplication with $B_i$.

Note that $\det(B_0-F_0B_0)\neq 0$ if and only if $F_0$ has no eigenvalue $1$, and if $\det(I-M_i(v)F_i)$ is independent of $v\in \ZZ^{k_0}$, then so is $\det(B_i-M_i(v)F_iB_i)$. Thus, it follows from Theorem \ref{thm:N-sol} that \[
N(f,q)=\prod_{i=0}^c \left|\det(B_i-F_iB_i)\right|=\prod_{i=0}^c \left|\det(I-F_i)\right|\left|\det(B_i)\right|.
\]
Note that we have \[
\prod_{i=0}^c\left|\det(B_i)\right|=\prod_{i=0}^c[\Lambda'_i:\Lambda_i]=[\Gamma':\Gamma].
\]
Indeed, for the second equality, we can write \[
[\Gamma':\Gamma]=[\Gamma'/\Gamma'_1:\Gamma\Gamma'_1/\Gamma'_1][\Gamma'_1:\Gamma\cap \Gamma'_1]=[\Lambda'_0:\Lambda_0][\Gamma'_1:\Gamma_1]
\]
and proceed inductively to find $[\Gamma':\Gamma]=\prod_{i=0}^c[\Lambda'_i:\Lambda_i]$.

Together, we get:

\begin{thm}\label{thm:N(f,q)}
Suppose $f:\Gamma\orb G\to \Gamma'\orb G$ is a map between special solvmanifolds, where $\Gamma'$ is net and $\Gamma$ is a finite index subgroup of $\Gamma'$. 
Let $\{F_i\in \QQ^{k_i\times k_i}\}_{i=0}^c$ be the set of linearisation matrices of $f$.
The Nielsen coincidence number of $f$ with the natural projection $q:\Gamma\orb G\to \Gamma'\orb G$ is given by
\begin{equation}\label{eq:coin-sol}
N(f,q)=[\Gamma':\Gamma]\prod_{i=0}^c \left|\det(I-F_i)\right|.
\end{equation}
\end{thm}

\begin{ex}We consider exactly the same solvmanifolds and maps as in Example~\ref{ex:first-formula}. By our choice of $A$, the group $\Gamma'$ is indeed net and $q: \Gamma \orb G \to 
\Gamma' \orb G$ is the natural projection, so the conditions of Theorem~\ref{thm:N(f,q)} are satisfied.

We had that $\Lambda_0 = 2 \ZZ$ and $\Lambda_0'= \ZZ$, so $B_0 =2$. The map $F_0$ is given by 
\[ F_0: 2 \ZZ \to \ZZ: 2t \mapsto -t\]
so as a $1 \times 1$--matrix $F_0=\begin{bmatrix} -\frac{1}{2} \end{bmatrix}$. For $F_1$, it holds that $\Lambda_1 = \Lambda_1'=\ZZ^2$ (so $B_1$ is the identity matrix) and $F_1:\ZZ^2 \to \ZZ^2 $ is the zero map. 
According to the theorem above, we have that 
\[ N(f,q) = [\Gamma' : \Gamma ]\cdot \left|\det(I -F_0) \right| \cdot \left|\det(I - F_1)\right|= 2 \cdot| 1-(-\tfrac{1}{2})|\cdot 1 = 3\]
which coincides with our earlier computation in Example~\ref{ex:first-formula}.
\end{ex}

\section{Nielsen numbers of $n$-valued maps on infra-solvmanifolds}

\subsection{Infra-solvmanifolds}

Many equivalent ways to define an infra-solvmanifold have appeared in the literature; for an overview, see \cite{kurokiyu}. We will use the definition taken e.g.\ in \cite{leelee2009}: let $G$ be a connected simply connected solvable Lie group, and let $\Aff(G)=G\rtimes \Aut(G)$ be the \emph{affine group} of $G$, which acts on $G$ as \[
\forall (a,A)\in \Aff(G),\,g\in G: (a,A)g=aA(g).
\] 
An infra-solvmanifold is a quotient $\pi\orb G$ where $\pi$ is a torsion-free subgroup of $\Aff(G)$ which is a finite extension of a lattice in $G$.

\begin{rmk}
From this definition of an infra-solvmanifold, it does not immediately follow that every solvmanifold is an infra-solvmanifold. This is, however, the case: in \cite{kurokiyu} it is shown that a (compact) infra-solvmanifold can equivalently be defined as a quotient $\Delta\orb G$ where $G$ is a connected simply connected solvable Lie group, and $\Delta$ is a torsion-free subgroup of $\Aff(G)$ acting cocompactly on $G$ such that $\text{hol}(\Delta)$, the projection of $\Delta$ in $\Aut(G)$, has compact closure in $\Aut(G)$ (this is Definition 5 in \cite{kurokiyu}). In particular, this is satisfied when $\Delta$ is a closed cocompact subgroup of $G$, since $G$ is torsion-free and $\text{hol}(\Delta)$ is trivial.
\end{rmk}

If $\Gamma$ is a lattice in $G$, the coset space $\Gamma\orb G$ is a special solvmanifold. Thus, by definition, every infra-solvmanifold admits a finite regular cover by a special solvmanifold. It follows from \cite[Theorem 4.8 \& Remark 4.9]{iris} that this cover can always be taken so that $\Gamma$ is net. On the other hand:

\begin{lem}\label{lem:ism-cover}
Let $\pi\orb G$ be an infra-solvmanifold, and $\Gamma\orb G$ a special solvmanifold that is a finite regular covering space of $\pi\orb G$. Then $\pi$ admits an $(f,\Gamma)$-invariant subgroup (see Theorem \ref{thm:RT}) for any $n$-valued map $f:\pi\orb G\to D_n(\pi\orb G)$.
\end{lem}

\begin{proof}
As we saw in Remark \ref{rmk:polycyclic}, if $\Gamma\orb G$ is a special solvmanifold, the group $\Gamma$ is always a polycyclic group, so $\pi$ is virtually polycyclic. Given an $n$-valued map $f:\pi\orb G\to D_n(\pi\orb G)$, define \[
S=\langle \alpha^{[\pi:\Gamma]} \mid \alpha\in \ker\sigma \rangle,
\]
where $\sigma:\pi\to \Sigma_n$ is the morphism induced by $f$.
By construction, $S$ is a normal subgroup of $\pi$ contained in $\Gamma$ and in all groups $S_j$, and $\varphi_j(S)\subseteq \Gamma$ for each $j$.
Also, the quotient $\pi/S$ is periodic: if $r=[\pi:\ker\sigma]$, then $\alpha^{r[\pi:\Gamma]}\in S$ for any $\alpha\in \pi$. Since periodic virtually polycyclic groups are automatically finite, it follows that $\pi/S$ is finite, i.e.\ $S$ is a finite index subgroup of $\pi$.
\end{proof}

Special solvmanifolds are always orientable (see \cite{HLP2011}). Thus, we can apply Theorem \ref{thm:RT} (with $M=\pi\orb G$, $\bar{M}=\Gamma\orb G$ and $\hat{M}=S\orb G$) to express Nielsen numbers of $n$-valued maps on infra-solvmanifolds in terms of Nielsen coincidence numbers of pairs $\bar{f}_j,q:S\orb G\to \Gamma\orb G$, where $\Gamma\orb G$ is a special solvmanifold, $\Gamma$ is net, $S$ is a finite index subgroup of $\Gamma$, and $q$ is the natural projection: the setting of Theorem \ref{thm:N(f,q)}.

To express the Nielsen number of an $n$-valued map $f$ of an infra-solvmanifold in a closed formula that can be computed directly from $f$, we need some more notations.

Let $\pi\orb G$ be an infra-solvmanifold, with $\Gamma\lhd_f \pi$ a lattice in $G$ (we use $\lhd_f$ to denote a finite index normal subgroup). Since $\Gamma\orb G$ is a special solvmanifold, the group $\Gamma$ is a strongly torsion-free $S$-group. If $\Gamma_1=\sqrt[\Gamma]{[\Gamma,\Gamma]}$ and $\Gamma_i=\sqrt[\Gamma_1]{\gamma_i(\Gamma_1)}$ are the associated groups from the previous sections, we get a series \[
\pi \rhd \Gamma=\Gamma_0 \rhd \Gamma_1 \rhd \ldots \rhd \Gamma_c \rhd \Gamma_{c+1}=1
\]
where $\pi/\Gamma_0$ is finite and $\Lambda_i=\Gamma_i/\Gamma_{i+1}$ is free abelian for all $i\geq 0$. In particular, the group $\pi$ is a torsion-free virtually polycyclic group.\footnote{In fact, as shown in \cite{baues2004}, a group $\pi$ can be realised as the fundamental group of an infra-solvmanifold \emph{if and only if} it is a torsion-free virtually polycyclic group, and (like for solvmanifolds) this group completely determines the infra-solvmanifold.}

For every $i\geq 0$, we have a well-defined morphism $A_i:\pi\to \Aut(\Lambda_i)$ given by \[
A_i(\alpha):\Lambda_i\to \Lambda_i:\gamma_i\Gamma_{i+1}\mapsto \alpha \gamma_i\alpha^{-1}\Gamma_{i+1}.
\]
We will call the collection $\{\Lambda_i,A_i \}_{i\geq 0}$ the \emph{linearisation of $\pi$ corresponding to $\Gamma$}. Note that, if $\{\Lambda_i,\mu_i \}_{i\geq 1}$ is the linearisation of $\Gamma$ as defined in section \ref{subsec:solvmanifolds}, then $\mu_i(\gamma\Gamma_1)=A_i(\gamma)$ for all $\gamma\in \Gamma$ and $i\geq 1$.

If $\pi$ and $\pi'$ are groups with strongly torsion-free $S$-groups $\Gamma\lhd_f \pi$ and $\Gamma'\lhd_f \pi'$, and $\varphi:\pi\to \pi'$ is a morphism such that $\varphi(\Gamma)\subseteq \Gamma'$, then we call the linearisation of the restriction $\varphi':\Gamma\to \Gamma'$ also the \emph{linearisation of $\varphi$ corresponding to $(\Gamma,\Gamma')$}. For later reference, we observe:

\begin{lem}\label{lem:fix-phii-phi}
Suppose $\pi\subseteq \pi'$ are torsion-free groups with strongly torsion-free $S$-groups $\Gamma\lhd_f \pi$ and $\Gamma'\lhd_f \pi'$, and $\varphi:\pi\to \pi'$ is a morphism such that $\varphi(\Gamma)\subseteq \Gamma'$, with linearisation $\{\varphi_i \}_{i=0}^c$. If $\fix(\varphi_i)=1$ for all $i\in \{0,\ldots,c\}$, then $\fix(\varphi)=1$.
\end{lem}

\begin{proof}
Let $\varphi':\Gamma\to \Gamma'$ be the restriction of $\varphi$. It suffices to prove that $\fix(\varphi')=1$: indeed, if this is the case, suppose $\alpha\in \fix(\varphi)$. Then $\alpha^{[\pi:\Gamma]}\in\fix(\varphi')$, so $\alpha^{[\pi:\Gamma]}=1$. Since $\pi$ is torsion-free, it follows that $\alpha=1$.

Suppose $\alpha\in \fix(\varphi')$. We show that $\alpha\in \Gamma_i$ by induction on $i$, with $\Gamma_0=\Gamma$. Suppose $\alpha\in \Gamma_i$ for some $i\geq 0$. Then $\alpha\Gamma_{i+1}\in \fix(\varphi_i)$. Since $\fix(\varphi_i)=1$, it follows that $\alpha\in \Gamma_{i+1}$. By induction, we eventually get $\alpha\in \Gamma_{c+1}=1$.
\end{proof}

For an $n$-valued map $f:\pi\orb G\to D_n(\pi\orb G)$ of an infra-solvmanifold $\pi\orb G$, let $\Gamma\orb G$ be a special solvmanifold that is a finite regular covering space of $\pi\orb G$, with $\Gamma$ net, and let $S$ be an $(f,\Gamma)$-invariant subgroup of $\pi$ (which exists by Lemma \ref{lem:ism-cover}). 
Let $\varphi_j:S_j\to \pi$ be the morphisms induced by (some lift of) $f$. For each $j$, we can consider the linearisation $\{\varphi_{i,j}\}_{i=0}^c$ of $\varphi_j$ corresponding to the subgroups $S\lhd_f S_j$ and $\Gamma\lhd_f \pi$. We will call the total collection $\{\varphi_{i,j} \}_{i=0;j=1}^{c;n}$ the \emph{linearisation of $f$ corresponding to $(S,\Gamma)$}.

We can again choose identifications $\Lambda_i\cong \ZZ^{k_i}$ for all $i$. Then we can consider the matrices $F_{i,j}\in \QQ^{k_i\times k_i}$ representing $\varphi_{i,j}$ as in the previous section.

\begin{df}
We will call $\{F_{i,j} \}_{i=0;j=1}^{c;n}$ the \emph{set of linearisation matrices of $f$ corresponding to $(S,\Gamma)$}. (Again, we keep in mind that these depend on the lift of $f$ chosen to compute the morphisms $\varphi_j$.)
\end{df}

We will also view $A_i$ as morphisms $\pi\to \GL_{k_i}(\ZZ)$, and still call $\{\ZZ^{k_i},A_i \}_{i=0}^c$ the linearisation of $\pi$ corresponding to $\Gamma$.

\subsection{Main result}

We can now express Nielsen numbers of $n$-valued maps on infra-solvmanifolds in terms of their linearisations:

\begin{thm}\label{thm:main}
Let $f:\pi\orb G\to D_n(\pi\orb G)$ be an $n$-valued map of an infra-solvmanifold $\pi\orb G$, finitely covered by a special solvmanifold $\Gamma\orb G$ whose fundamental group $\Gamma$ is net. Then \begin{equation}\label{eq:mainthm}
N(f)=\frac{1}{[\pi:\Gamma]} \sum_{j=1}^n \sum_{\bar{\alpha}\in \pi/\Gamma} \prod_{i=0}^c \left|\det(I-A_i(\alpha)F_{i,j})\right|,
\end{equation}
where $\{\ZZ^{k_i},A_i\}_{i=0}^c$ is the linearisation of $\pi$ corresponding to $\Gamma$, and $\{F_{i,j} \}_{i=0;j=1}^{c;n}$ is the set of linearisation matrices of $f$ corresponding to $(S,\Gamma)$, where $S$ is any $(f,\Gamma)$-invariant subgroup of $\pi$.
\end{thm}

\begin{proof}
We will apply Theorem \ref{thm:RT}. Let $\bar{f}_1,\ldots,\bar{f}_n:S\orb G\to \Gamma\orb G$ be the maps from that theorem, and $q:S\orb G\to \Gamma\orb G$ the projection. In \cite{RT}, it is observed that the morphism induced by $\bar{\alpha}\bar{f}_j$ is $\mu(\alpha)\varphi'_j:S\to \Gamma$, for some representative $\alpha\in \pi$ for $\bar{\alpha}$ that depends on the chosen lift for $\bar{\alpha}\bar{f}_j$. Note that the linearisation of $\mu(\alpha):\Gamma\to \Gamma$ is $\{A_i(\alpha) \}_{i=0}^c$, so the set of linearisation matrices of $\bar{\alpha}\bar{f}_j$ is $\{A_i(\alpha)F_{i,j}\}_{i=0}^c$. It follows from Theorem \ref{thm:N(f,q)} that \begin{equation}\label{eq:N(afj,q)}
N(\bar{\alpha}\bar{f}_j,q)=[\Gamma:S]\prod_{i=0}^c \left|\det(I-A_i(\alpha)F_{i,j})\right|.
\end{equation}
It remains to be shown that if this expression is non-zero, then $\fix(\mu(\alpha)\varphi_j)=1$.

Suppose the expression is non-zero. Then $\det(I-A_i(\alpha)F_{i,j})\neq 0$ for all $i$. It follows that $\fix(A_i(\alpha)\varphi_{i,j})=1$ for all $i$ (a non-trivial fixed point of $A_i(\alpha)\varphi_{i,j}$ corresponds to a non-trivial solution $v\in \ZZ^{k_i}$ to $(I-A_i(\alpha)F_{i,j})v=0$). By Lemma \ref{lem:fix-phii-phi}, this implies that $\fix(\mu(\alpha)\varphi_j)=1$ as well.
\end{proof}

\subsection{Remarks}

First, we observe that the right hand side of \eqref{eq:mainthm} depends on some choices, and should be independent of these choices for the formula to make sense.

\begin{itemize}
\item By Remarks \ref{rmk:det-indep} and \ref{rmk:N-sol-indep}, the formula \eqref{eq:N(afj,q)} (and therefore also the final formula) is independent of the choice of identifications $\Lambda_i\cong \ZZ^{k_i}$ and of the choice of lift for $\bar{\alpha}\bar{f}_j$, so of the choice of representative $\alpha$ for $\bar{\alpha}$.

\item The morphisms $\varphi_j$ induced by $f$ (and therefore the linearisation of $f$) depend on the chosen lift $\tilde{f}$. The same argument as in \cite[p.\,16]{RT} shows that the formula is independent of this choice of lift.

\item To show that the formula does not depend on the chosen $(f,\Gamma)$-invariant subgroup $S$, it suffices to show the following lemma. Indeed, if $S_1$ and $S_2$ are two $(f,\Gamma)$-invariant subgroups, then so is $S_0=S_1\cap S_2$, and the lemma implies that both the linearisations corresponding to $S_1$ and $S_2$ are equal to the one corresponding to $S_0$.
\end{itemize}

\begin{lem}\label{lem:lin-invar}
If $S'\subseteq S$ are $(f,\Gamma)$-invariant subgroups, the linearisation morphisms of $f$ corresponding to $(S,\Gamma)$ are given by the same matrices as those corresponding to $(S',\Gamma)$.
\end{lem}

\begin{proof}
Write $S_0=S$, $S_1=\sqrt[S]{[S,S]}$ and $S_i=\sqrt[S]{\gamma_i(S)}$ for $i\geq 1$. By Lemmas \ref{lem:gamma-fi} and \ref{lem:gamma-i-fi}, for each $i\geq 0$, the quotient $S_i/S_{i+1}$ can be identified with a finite index subgroup of $\Gamma_i/\Gamma_{i+1}\cong \ZZ^{k_i}$, say $B_i\ZZ^{k_i}$. In turn, for $S'$, each quotient $S'_i/S'_{i+1}$ can be identified with a finite index subgroup of $S_i/S_{i+1}\cong B_i\ZZ^{k_i}$, say $B'_iB_i\ZZ^{k_i}$.

The linearisation morphisms of $f$ corresponding to $(S,\Gamma)$ can be viewed as morphisms $B_i\ZZ^{k_i}\to \ZZ^{k_i}$, and those corresponding to $(S',\Gamma)$ are their restrictions $B'_iB_i\ZZ^{k_i}\to \ZZ^{k_i}$, which are given by the same matrices.
\end{proof}

Next, we argue that the known formulas for the Nielsen number of $n$-valued maps on infra-nilmanifolds \cite{RT} and for the Nielsen number of single-valued maps on infra-solvmanifolds \cite{iris} follow from our result.

\subsubsection*{Special case 1: infra-nilmanifolds}

An infra-nilmanifold is essentially a quotient $\pi\orb N$, where $N$ is a connected simply connected nilpotent Lie group, and $\pi$ is a torsion-free subgroup of $\Aff(N)$ which is a finite extension of a lattice $\Gamma$ in $N$. The result of \cite{RT} is:

\begin{thm}[{\cite[Theorem 7.3]{RT}}]
Let $f:\pi\orb N\to D_n(\pi\orb N)$ be an $n$-valued map on an infra-nilmanifold $\pi\orb N$, finitely covered by a nilmanifold $\Gamma\orb N$. Let $S$ be an $(f,\Gamma)$-invariant subgroup of $\pi$, and $\varphi'_j:S\to \Gamma$ the morphisms induced by $f$. Then \[
N(f)=\frac{1}{[\pi:\Gamma]}\sum_{j=1}^{n}\sum_{\bar{\alpha}\in \pi/\Gamma} \,|\!\det(I-(\mu(\alpha)\varphi'_j)_*)|
\]
(where the star denotes the Lie algebra morphism $\n\to\n$ induced by the unique extension $N\to N$, and $I:\n\to\n$ is the identity Lie algebra morphism).
\end{thm}

To view this as a special case of Theorem \ref{thm:main}, note that in this setting, the $i=0$ term of the product in \eqref{eq:mainthm} is trivial. The remaining product \[
\prod_{i=1}^c \left|\det(I-A_i(\alpha)F_{i,j})\right|
\]
is equal to $|\!\det(I-(\mu(\alpha)\varphi'_j)_*)|$ by the following result, which follows from \cite[Theorem 5.4]{charlotte2}:

\begin{lem}
Let $\Gamma'$ be a finitely generated torsion-free nilpotent group and $\Gamma$ a finite index subgroup. For each $i\geq 1$, choose an identification $\Gamma'_i/\Gamma'_{i+1}\cong \ZZ^{k_i}$, and let $B_i\ZZ^{k_i}$ be the finite index subgroup corresponding to $\Gamma_i/\Gamma_{i+1}$. Let $N$ be the Mal'cev completion of $\Gamma'$, with Lie algebra $\n$.

For a morphism $\varphi:\Gamma\to \Gamma'$, let $\varphi_*:\n\to\n$ be the Lie algebra morphism induced by $\varphi$, and let $F_i\in \QQ^{k_i\times k_i}$ be the matrices such that the induced morphisms $\varphi_i:B_i\ZZ^{k_i}\to\ZZ^{k_i}$ are given by multiplication with $F_i$. Then
\[
\det (I-\varphi_*) = \prod_{i=1}^{c}\det (I-F_i).
\]
\end{lem}

\subsubsection*{Special case 2: single-valued maps}

The result of \cite{iris} is:

\begin{thm}[{\cite[Theorem 5.2]{iris}}]
Let $f:\pi\orb G\to \pi\orb G$ be a self-map of an infra-solvmanifold $\pi\orb G$, finitely covered by a special solvmanifold $\Gamma\orb G$ whose fundamental group $\Gamma$ is net. Let $S\lhd_f\pi$ be a fully invariant subgroup contained in $\Gamma$. Then \begin{equation}\label{eq:iris}
N(f)=\frac{1}{[\pi:\Gamma]} \sum_{\bar{\alpha}\in \pi/\Gamma} \prod_{i=0}^c \left|\det(I-A_i(\alpha)F_i)\right|,
\end{equation}
where $\{\ZZ^{k_i},A_i\}_{i=0}^c$ and $\{F_i \}_{i=0}^{c}$ are the respective linearisations of $\pi$ and $f$ corresponding to $S$.
\end{thm}

Here, the linearisation of the single-valued map $f$ corresponding to a fully invariant subgroup $S$ means the linearisation of the restriction $S\to S$ of the morphism $\varphi:\pi\to \pi$ induced by $f$.

\begin{rmk}
The reason why linearisations in \cite{iris} are taken with respect to a fully invariant subgroup is that, in that paper, linearisations are only defined for morphisms with the same domain and image. Indeed, for their purposes it only makes sense to define it in that setting, since they study fixed points of $f$ using \emph{fixed points} of lifts to the solvmanifold, rather than coincidences. That way they could use the existing formula \eqref{eq:KeppelmannMcCord} of Keppelmann and McCord, instead of proving the generalisation \eqref{eq:coin-sol} like we did.
\end{rmk}

To show that the formula \eqref{eq:iris} follows from our result, note that in the single-valued case, an $(f,\Gamma)$-invariant subgroup $S$ of $\pi$ is a subgroup $S\lhd_f\pi$ contained in $\Gamma$ such that $\varphi(S)\subseteq \Gamma$, where $\varphi:\pi\to\pi$ is the morphism induced by $f$. Certainly, this is the case if $S\lhd_f\pi$ is a fully invariant subgroup contained in $\Gamma$. Thus, it suffices to show that for such $S$,
\begin{itemize}
\item[(i)] the linearisation morphisms of $\pi$ corresponding to $\Gamma$ are given by the same matrices as those corresponding to $S$,
\item[(ii)] the linearisation morphisms of $f$ corresponding to $(S,\Gamma)$ are given by the same matrices as those corresponding to $S$.
\end{itemize}

\begin{proof}
As in the proof of Lemma \ref{lem:lin-invar}, each of the quotients $S_i/S_{i+1}$ can be identified with a finite index subgroup of $\Gamma_i/\Gamma_{i+1}\cong \ZZ^{k_i}$, say $B_i\ZZ^{k_i}$. Under these identifications, if $A_i:\pi\to \GL_{k_i}(\ZZ)$ are the linearisation morphisms of $\pi$ corresponding to $\Gamma$, those corresponding to $S$ are the morphisms $A'_i:\pi\to \GL(B_i\ZZ^{k_i})$ where $A'_i(\alpha):B_i\ZZ^{k_i}\to B_i\ZZ^{k_i}$ is given by the same matrix as $A_i(\alpha)$.

Similarly, the linearisation morphisms of $f$ corresponding to $S$ can be viewed as morphisms $B_i\ZZ^{k_i}\to B_i\ZZ^{k_i}$, and those corresponding to $(S,\Gamma)$ as the same morphisms viewed as $B_i\ZZ^{k_i}\to \ZZ^{k_i}$, so they are also given by the same matrices.
\end{proof}

\subsection{Example}

Consider the matrix \[
A=\begin{bmatrix}
0 & 1 & 0 & 0 \\
0 & 0 & 1 & 0 \\
0 & 0 & 0 & 1 \\
-1 & 1 & 1 & 1
\end{bmatrix}
\in \GL_4(\ZZ).
\]
This matrix was already used in \cite[p.\ 504, Example 2]{dekimpeleeraymond} and in \cite[Example 4.14]{iris}. The eigenvalues of $A$ are $\alpha$, $\alpha^{-1}$, $\beta$ and $\beta^{-1}$ where $\alpha\in \RR_{>0}\setminus \{1\}$ and $\beta=e^{i\theta}$ is not a root of unity.

Take $P\in \GL_4(\RR)$ such that \[
A=P \begin{bmatrix}
\alpha & 0 & 0 & 0 \\
0 & \alpha^{-1} & 0 & 0 \\
0 & 0 & \cos\theta & \sin\theta \\
0 & 0 & -\sin\theta & \cos\theta
\end{bmatrix} P^{-1}.
\]
Then, for $t\in \RR$, we can define $A^t \in \GL_4(\RR)$ as \[
A^t= P \begin{bmatrix}
\alpha^t & 0 & 0 & 0 \\
0 & \alpha^{-t} & 0 & 0 \\
0 & 0 & \cos(t\theta) & \sin(t\theta) \\
0 & 0 & -\sin(t\theta) & \cos(t\theta)
\end{bmatrix} P^{-1}.
\]

Now define \[
B=\begin{bmatrix}
-I_4 & 0 & 0 \\
0 & A & 0 \\
0 & 0 & A^3
\end{bmatrix}
\in \ZZ^{12\times 12},
\]
where the $0$'s denote $4\times 4$ zero matrices, and $I_4$ is the $4\times 4$ identity matrix. Note that $\det B=1$. For $t\in \RR$, we can define \[
B^t= \begin{bmatrix}
* & 0 & 0 \\
0 & A^t & 0 \\
0 & 0 & A^{3t}
\end{bmatrix}
\]
where \[
*=\begin{bmatrix}
\cos(\pi t) & \sin(\pi t) & 0 & 0 \\
-\sin(\pi t) & \cos(\pi t) & 0 & 0 \\
0 & 0 & \cos(\pi t) & \sin(\pi t) \\
0 & 0 & -\sin(\pi t) & \cos(\pi t)
\end{bmatrix}.
\]
We obtain a solvable Lie group $G=\RR^{12}\rtimes_\psi \RR$ with $\psi:\RR\to \GL_{12}(\RR):t\mapsto B^t$. 

Consider the lattice $\Gamma=\ZZ^{12}\rtimes_\psi \ZZ$ in $G$. We can identify $\Lambda_1=\Gamma_1=\ZZ^{12}\rtimes 0$ and $\Lambda_0=0^{12}\rtimes\ZZ$, and the morphism $\mu_1:\Lambda_0\to \Aut(\Lambda_1)$ is given by \[
\mu_1:0^{12}\rtimes\ZZ\to \GL(\ZZ^{12}\rtimes 0):(\bar{0},z)\mapsto (B^z,0).
\]
Since $-1$ is an eigenvalue of $B$, the group $\Gamma$ is not net. Thus, $\Gamma\orb G$ is a special solvmanifold with non-net fundamental group; in particular, an infra-solvmanifold.

We can define a $3$-valued map $f:\Gamma\orb G\multimap \Gamma\orb G$ by \[
\Gamma (\bar{x},y)\mapsto \{\Gamma(F\bar{x},\tfrac{y}{3}), \Gamma(FB\bar{x},\tfrac{y}{3}+\tfrac{1}{3}), \Gamma(FB^2\bar{x},\tfrac{y}{3}+\tfrac{2}{3})  \}
\]
where \[
F=\begin{bmatrix}
-I_4 & 0 & 0 \\
0 & 0 & 0 \\
0 & I_4 & 0
\end{bmatrix} \in \ZZ^{12\times 12}.
\]
Using the fact that $FB^3=BF$, one easily verifies that $f$ is well-defined with induced morphism \[
\begin{split}
\varphi: \Gamma \to \Gamma^3 \rtimes \Sigma_3:
(\bar{z},0) &\mapsto ((F\bar{z},0),\, (FB\bar{z},0),\, (FB^2\bar{z},0);\, \id ) \\
(\bar{0},1) &\mapsto ((\bar{0},0),\, (\bar{0},0),\, (\bar{0},1);\, (123)^{-1})
\end{split}
\]
(where $(123)\in \Sigma_3$ is the permutation sending $1$ to $2$, $2$ to $3$ and $3$ to $1$). 
We will now compute $N(f)$ using Theorem \ref{thm:main}.

As a finite index net subgroup of $\Gamma$, we can take $\Gamma_0=\ZZ^{12}\rtimes_\psi 2\ZZ$. Indeed, for this group we have $\Lambda_1=0^4\times\ZZ^8\rtimes 0$ and $\Lambda_0=\ZZ^4\times 0^8\rtimes 2\ZZ$, and \[
\mu_1:\ZZ^4\times 0^8\rtimes 2\ZZ \to \GL(0^4\times\ZZ^8\rtimes 0):(\bar{w},2z)\mapsto (B^{2z},0).
\]
All possible eigenvalues of $\mu_1(\bar{w},2z)$ are integer powers of eigenvalues of $B^2$, so integer powers of $1$, $\alpha$ and $\beta$. The group generated by these powers, $\langle\alpha,\beta\rangle$, contains no non-trivial roots of unity: if $\alpha^k\beta^\ell=1$, then $1=|\alpha^k\beta^\ell|=|\alpha|^k$. Since $|\alpha|\neq 1$, it follows that $k=0$. But then $\beta^\ell=1$, which implies $\ell=0$ since $\beta$ is no root of unity.

As $(f,\Gamma_0)$-invariant subgroup, we can take $S=\ZZ^{12}\rtimes_\psi 6\ZZ$, with $\Lambda'_1=0^4\times\ZZ^8\rtimes 0$ and $\Lambda'_0=\ZZ^4\times 0^8\rtimes 6\ZZ$. The linearisation of $f$ corresponding to $(S,\Gamma_0)$ is given by \[
\begin{split}
\varphi_{0i}: \ZZ^4\times 0^8\rtimes 6\ZZ &\to \ZZ^4\times 0^8\rtimes 2\ZZ: \\
\left(\begin{bmatrix} \bar{z} \\ \bar{0} \\ \bar{0} \end{bmatrix},0\right) &\mapsto \left(FB^{i-1}\begin{bmatrix} \bar{z} \\ \bar{0} \\ \bar{0} \end{bmatrix},0\right)=\left(\begin{bmatrix} (-1)^i\bar{z} \\ \bar{0} \\ \bar{0} \end{bmatrix},0\right) \\
(\bar{0},6z) & \mapsto (\bar{0},2z) \\
\varphi_{1i}: 0^4\times \ZZ^8\rtimes 0 &\to 0^4\times \ZZ^8\rtimes 0: \\
\left(\begin{bmatrix} \bar{0} \\ \bar{z}_1 \\ \bar{z}_2 \end{bmatrix},0\right) &\mapsto \left(FB^{i-1}\begin{bmatrix} \bar{0} \\ \bar{z}_1 \\ \bar{z}_2 \end{bmatrix},0\right)=\left(\begin{bmatrix} \bar{0} \\ \bar{0} \\ A^{i-1}\bar{z}_1 \end{bmatrix},0\right)
\end{split}
\]
for $i=1,2,3$. The linearisation of $\Gamma$ corresponding to $\Gamma_0$ is \[
\begin{split}
A_0(\bar{u},w): \ZZ^4\times 0^8\rtimes 2\ZZ &\to \ZZ^4\times 0^8\rtimes 2\ZZ: \\
\left(\begin{bmatrix} \bar{z} \\ \bar{0} \\ \bar{0} \end{bmatrix},v\right) &\mapsto \left(B^w\begin{bmatrix} \bar{z} \\ \bar{0} \\ \bar{0} \end{bmatrix},v\right)=\left(\begin{bmatrix} (-1)^w\bar{z} \\ \bar{0} \\ \bar{0} \end{bmatrix},v\right) \\
A_1(\bar{u},w): 0^4\times \ZZ^8\rtimes 0 &\to 0^4\times \ZZ^8\rtimes 0: \\
\left(\begin{bmatrix} \bar{0} \\ \bar{z}_1 \\ \bar{z}_2 \end{bmatrix},0\right) &\mapsto \left(B^w\begin{bmatrix} \bar{0} \\ \bar{z}_1 \\ \bar{z}_2 \end{bmatrix},0\right)=\left(\begin{bmatrix} \bar{0} \\ A^w\bar{z}_1 \\ A^{3w}\bar{z}_2 \end{bmatrix},0\right).
\end{split}
\]
Using the natural isomorphisms $\Lambda_1\cong \ZZ^8$ and $\Lambda_0\cong \ZZ^5$, we get: \begin{align*}
F_{01} &= \begin{bmatrix} -I_4 & 0 \\ 0 & \frac{1}{3} \end{bmatrix} \in \ZZ^{5\times 5} & F_{02} &= \begin{bmatrix} I_4 & 0 \\ 0 & \frac{1}{3} \end{bmatrix} \in \ZZ^{5\times 5} & F_{03} &= \begin{bmatrix} -I_4 & 0 \\ 0 & \frac{1}{3} \end{bmatrix} \in \ZZ^{5\times 5} \\
F_{11} &= \begin{bmatrix} 0 & 0 \\ I_4 & 0 \end{bmatrix} \in \ZZ^{8\times 8} & F_{12} &= \begin{bmatrix} 0 & 0 \\ A & 0 \end{bmatrix} \in \ZZ^{8\times 8} & F_{13} &= \begin{bmatrix} 0 & 0 \\ A^2 & 0 \end{bmatrix} \in \ZZ^{8\times 8} 
\end{align*}
\begin{align*}
A_0&:\Gamma\to \GL_5(\ZZ): (\bar{u},w) \mapsto \begin{bmatrix} (-1)^wI_4 & 0 \\ 0 & 1 \end{bmatrix} \\
A_1&:\Gamma\to \GL_8(\ZZ): (\bar{u},w) \mapsto \begin{bmatrix} A^w & 0 \\ 0 & A^{3w} \end{bmatrix}.
\end{align*}
Plugging everything into formula \eqref{eq:mainthm} (with representatives $\alpha=(\bar{0},0),(\bar{0},1)$ for $\Gamma/\Gamma_0$) gives 
\begin{align*}
N(f) &= \frac{1}{2} \big(\left|\det(I_5-A_0(\bar{0},0)F_{01})\det(I_8-A_1(\bar{0},0)F_{11})\right| \\ 
&\qquad + \left|\det(I_5-A_0(\bar{0},1)F_{01})\det(I_8-A_1(\bar{0},1)F_{11})\right| \\ 
&\qquad + \left|\det(I_5-A_0(\bar{0},0)F_{02})\det(I_8-A_1(\bar{0},0)F_{12})\right| \\ 
&\qquad + \left|\det(I_5-A_0(\bar{0},1)F_{02})\det(I_8-A_1(\bar{0},1)F_{12})\right| \\
&\qquad + \left|\det(I_5-A_0(\bar{0},0)F_{03})\det(I_8-A_1(\bar{0},0)F_{13})\right| \\ 
&\qquad + \left|\det(I_5-A_0(\bar{0},1)F_{03})\det(I_8-A_1(\bar{0},1)F_{13})\right| \big) \\
&= \frac{1}{2} \Bigg(\left|\det\left(I_5-\begin{bmatrix} -I_4 & 0 \\ 0 & \frac{1}{3} \end{bmatrix}\right)\det\left(I_8-\begin{bmatrix} 0 & 0 \\ I_4 & 0 \end{bmatrix}\right)\right| \\ 
&\qquad + \left|\det\left(I_5-\begin{bmatrix} I_4 & 0 \\ 0 & \frac{1}{3} \end{bmatrix}\right)\det\left(I_8-\begin{bmatrix} 0 & 0 \\ A^3 & 0 \end{bmatrix}\right)\right| \\
&\qquad + \left|\det\left(I_5-\begin{bmatrix} I_4 & 0 \\ 0 & \frac{1}{3} \end{bmatrix}\right)\det\left(I_8-\begin{bmatrix} 0 & 0 \\ A & 0 \end{bmatrix}\right)\right| \\ 
&\qquad + \left|\det\left(I_5-\begin{bmatrix} -I_4 & 0 \\ 0 & \frac{1}{3} \end{bmatrix}\right)\det\left(I_8-\begin{bmatrix} 0 & 0 \\ A^4 & 0 \end{bmatrix}\right)\right| \\
&\qquad + \left|\det\left(I_5-\begin{bmatrix} -I_4 & 0 \\ 0 & \frac{1}{3} \end{bmatrix}\right)\det\left(I_8-\begin{bmatrix} 0 & 0 \\ A^2 & 0 \end{bmatrix}\right)\right| \\ 
&\qquad + \left|\det\left(I_5-\begin{bmatrix} I_4 & 0 \\ 0 & \frac{1}{3} \end{bmatrix}\right)\det\left(I_8-\begin{bmatrix} 0 & 0 \\ A^5 & 0 \end{bmatrix}\right)\right| \Bigg) \\
&=\frac{1}{2}\cdot \frac{2}{3}\left(2^4+0+0+2^4+2^4+0 \right)=2^4=16.
\end{align*}


\end{document}